\documentclass[reqno]{amsart}
\usepackage{amsmath,amsthm,amsfonts,amssymb,verbatim,color}
\usepackage{color}

\numberwithin{equation}{section}

\theoremstyle{definition}
\newtheorem{thm}{Theorem}[section]
\newtheorem{lem}[thm]{Lemma}

\newtheorem{rem}[thm]{Remark}
\newtheorem{prop}[thm]{Proposition}
\newtheorem{ques}[thm]{Question}
\newtheorem{cor}[thm]{Corollary}

\def \e {\epsilon}

\def \E {\mathbb{E}}
\def \N {\mathbb{N}}
\def \Z {\mathbb{Z}}

\thanks{The second author was supported by NSF under grant DMS-1500575.}

\begin{document}
	
\title{Under- and over-independence in measure preserving systems}

\author{Terry Adams}
\address{U.S. Government, 9800 Savage Road, Ft. Meade, MD 20755} \email{terry@ganita.org}

\author{Vitaly Bergelson}
\address{Department of Mathematics, The Ohio State University, 231 West 18th Avenue, Columbus OH, 43210-1174, USA} \email{bergelson.1@osu.edu}

\author{Wenbo Sun}
\address{Department of Mathematics, The Ohio State University, 231 West 18th Avenue, Columbus OH, 43210-1174, USA}
\email{sun.1991@osu.edu}

\date{\today}

\begin{abstract}
	We introduce the notions of over- and under-independence for weakly mixing and (free) ergodic measure preserving actions and establish new results which complement and extend the theorems obtained in \cite{BFW} and \cite{Adams}. Here is a sample of results obtained in this paper:
	\begin{itemize}
		\item (Existence of density-1 UI and OI set) Let $(X,\mathcal{B},\mu,T)$ be an invertible probability measure preserving weakly mixing system. Then for any $d\in\mathbb{N}$, any non-constant integer-valued polynomials $p_{1},p_{2},\dots,p_{d}$ such that $p_{i}-p_{j}$ are also non-constant for all $i\neq j$,
		
		(i) there is $A\in\mathcal{B}$ such that the set
		$$\{n\in\mathbb{N}\colon\mu(A\cap T^{p_{1}(n)}A\cap\dots\cap T^{p_{d}(n)}A)<\mu(A)^{d+1}\}$$
		is of density 1.
		
		(ii) there is $A\in\mathcal{B}$ such that the set
		$$\{n\in\mathbb{N}\colon\mu(A\cap T^{p_{1}(n)}A\cap\dots\cap T^{p_{d}(n)}A)>\mu(A)^{d+1}\}$$
		is of density 1. 
		\item (Existence of Ces\`aro OI set) Let $(X,\mathcal{B},\mu,T)$ be a free, invertible, ergodic probability measure preserving system and $M\in\mathbb{N}$. 
		Then
		there is $A\in\mathcal{B}$ such that 
		$$\frac{1}{N}\sum_{n=M}^{N+M-1}\mu(A\cap T^{n}A)>\mu(A)^{2}$$
		for all $N\in\mathbb{N}$.
	\item
		(Nonexistence of Ces\`aro UI set) Let $(X,\mathcal{B},\mu,T)$ be an invertible probability  measure preserving system. For any measurable set $A$ satisfying $\mu(A) \in (0,1)$, 
		there exist infinitely many $N \in \N$ such that 
		\[
		\frac{1}{N} \sum_{n=0}^{N-1} \mu ( A \cap T^{n}A) > \mu(A)^2 . 
		\]
	\end{itemize}
\end{abstract}

\maketitle

\section{introduction}
The classical Poincar\'e Recurrence Theorem \cite{P} states that for any probability measure preserving system $(X,\mathcal{B},\mu,T)$, any $A\in\mathcal{B}$, and almost every $x\in A$, there exists $n\in\mathbb{N}$ such that $T^{n}x\in A$. \footnote{In this paper we will, as a rule, assume that the measure spaces we deal with are standard, that is, isomorphic mod 0  to a disjoint union of a finite number of atoms with an interval equipped with the Lebesgue measure.} This result is derived in \cite{P} from the fact (usually also called  Poincar\'e Recurrence Theorem)  that if $\mu(A)>0$, then $\mu(A\cap T^{-n}A)>0$ for some $n\in\mathbb{N}$. The correlation sequence $\mu(A\cap T^{-n}A), n\in\mathbb{N}$ is one of the most basic objects of ergodic theory. For example, the classical notions of ergodicity, weak mixing and mixing can be formulated as follows:
\begin{itemize}
	\item $(X,\mathcal{B},\mu,T)$ is \emph{ergodic} if and only if for all $A\in\mathcal{B}$,
	\begin{equation}\label{intro:1}
	\begin{split}
	\lim_{N\to\infty}\frac{1}{N}\sum_{n=0}^{N-1}\mu(A\cap T^{-n}A)=\mu(A)^{2};
	\end{split}
	\end{equation}
	\item $(X,\mathcal{B},\mu,T)$ is \emph{weakly mixing} if and only if for all $A\in\mathcal{B}$, there exists a set $E\subseteq\mathbb{N}$ with $d(E):=\lim_{N\to\infty}\frac{\vert E\cap\{0,1,\dots,N-1\}\vert}{N}=1$ such that
	\begin{equation}\label{intro:2}
		\begin{split}
		\lim_{n\to\infty,n\in E}\mu(A\cap T^{-n}A)=\mu(A)^{2};
		\end{split}
	\end{equation}
	\item $(X,\mathcal{B},\mu,T)$ is \emph{mixing} if and only if for all $A\in\mathcal{B}$,
		\begin{equation}\label{intro:3}
		\begin{split}
		\lim_{n\to\infty}\mu(A\cap T^{-n}A)=\mu(A)^{2}.
		\end{split}
		\end{equation}
\end{itemize}
While in each of  the above formulas the limit on the right is $\mu(A)^{2}$, it is apriori not clear whether the quantities in the left parts of the formulas may stay for all $n\neq 0$ below or above this limit. The following question was asked by the second author in \cite{BI}:
\begin{ques}\label{q50}
Is it true that for any invertible probability measure preserving mixing system $(X,\mathcal{B},\mu,T)$, there exists $A\in\mathcal{B}$ with $\mu(A)>0$ such that for all $n\neq 0$, $\mu(A\cap T^{n}A)<\mu(A)^{2}$? How about the reverse inequality $\mu(A\cap T^{n}A)>\mu(A)^{2}$?
\end{ques}
We will be referring to the phenomena alluded to in the above question as under- and over-independence (and use the abbreviation "UI" and "OI" when dealing with sets possessing these properties).\footnote{In \cite{BFW}, under- and over-independence are called under- and over-recurrence.} After staying dormant for about 20 years, the subject of under- and over-independence came to life in the recent paper \cite{BFW} where the authors showed that
\begin{itemize}
	\item not all mixing systems have UI sets;
	\item all ergodic systems with positive entropy have UI sets;
	\item there exist mixing systems which have both UI and OI sets.
\end{itemize} 
In \cite{Adams}, it was shown that actually every mixing system has an OI set; it is also proved in \cite{Adams} by a method different from that in \cite{BFW} that not every mixing system has a UI set. Analyzing the above results, one arrives at the natural conclusion that over-independence occurs more readily than under-independence. In spite of this trend, a positive result for under-independence is obtained, when it is shown that every weakly mixing system has density-1 UI sets. Thus, we are motivated by improving our intuition for under- and over-independence, as well as expanding results from the classic $\mathbb{Z}$-action case to more general situations. 
\subsection{Under- and over-independence for weakly mixing systems}
First of all, it is natural to inquire whether appropriately modified versions of under- and over-independence hold for weakly mixing systems. Taking into account the natural mode of convergence to independence in weakly mixing systems (see (\ref{intro:2}) above), we have the following analogue of Question \ref{q50}:
\begin{ques}
	Let $(X,\mathcal{B},\mu,T)$ be an invertible probability measure preserving weakly mixing system.
	
	(i) Is there a set $A\in\mathcal{B}$ with $\mu(A)>0$ such that for some $E\subseteq\mathbb{N}$ with $d(E)=1$, we have that $\mu(A\cap T^{n}A)<\mu(A)^{2}$ for all $n\in E$?
	
	(ii) Is there a set $A\in\mathcal{B}$ with $\mu(A)>0$ such that for some $E\subseteq\mathbb{N}$ with $d(E)=1$, we have that $\mu(A\cap T^{n}A)>\mu(A)^{2}$ for all $n\in E$?   
\end{ques} 	
	We show in this paper that the answers to both (i) and (ii) are YES. 
	Moreover,
	we obtain a general result pertaining to under- and over-independence
	for multiple recurrence in weakly mixing systems. We formulate first a relevant version of the polynomial weakly mixing theorem which was obtained in \cite{BPet}:
	\begin{thm}[\cite{BPet}]\label{thm:BP}
		An invertible probability measure preserving system $(X,\mathcal{B},\mu,T)$ is weakly mixing if and only if for any $d\in\mathbb{N}$, any non-constant integer-valued polynomials $p_{1},\dots,p_{d}$ such that $p_{i}-p_{j}$ are also non-constant for all $i\neq j$, and any $A\in\mathcal{B}$, there exists a set $E\subseteq\mathbb{N}$ with $d(E)=1$ such that
		\begin{equation}\nonumber
		\begin{split}
		\lim_{n\to\infty,n\in E}\mu(A\cap T^{p_{1}(n)}A\cap\dots\cap T^{p_{d}(n)}A)=\mu(A)^{d+1}.
		\end{split}
		\end{equation}
	\end{thm}
Here is now the formulation of our result pertaining to over- and under-independence in weakly mixing systems:
		\begin{thm} [Existence of density-1 UI and OI set]
			\label{dense-one-under}
			Let $(X,\mathcal{B},\mu,T)$ be an invertible probability measure preserving weakly mixing system. Then for any $d\in\mathbb{N}$, any non-constant integer-valued polynomials $p_{1},p_{2},\dots,p_{d}$ such that $p_{i}-p_{j}$ are also non-constant for all $i\neq j$,
			
			(i) there is $A\in\mathcal{B}$ such that the set
			$$\{n\in\mathbb{N}\colon\mu(A\cap T^{p_{1}(n)}A\cap\dots\cap T^{p_{d}(n)}A)<\mu(A)^{d+1}\}$$
			is of density 1;
			
			(ii) there is $A\in\mathcal{B}$ such that the set
			$$\{n\in\mathbb{N}\colon\mu(A\cap T^{p_{1}(n)}A\cap\dots\cap T^{p_{d}(n)}A)>\mu(A)^{d+1}\}$$
			is of density 1. \footnote{
				In fact, the polynomial functions $p_{1}(n),\dots,p_{d}(n)$ in Part (ii) of Theorem \ref{dense-one-under} can be replaced with any functions $a_{1}(n),\dots,a_{d}(n)$ satisfying the "multiple weakly mixing  theorem", meaning that for all $A\in\mathcal{B}$, there exists $E\subseteq\mathbb{N}$ with $d(E)=1$ such that
				$$\lim_{n\to\infty,n\in E}\mu(A\cap T^{a_{1}(n)}A\cap\dots\cap T^{a_{d}(n)}A)=\mu(A)^{d+1}.$$
				In particular, one can take $a_{i}(n)$ to be tempered functions or functions from a Hardy field. See \cite{BH}. 									
			}
		\end{thm}

	Part (i) of Theorem \ref{dense-one-under} is proved in Section \ref{sec:3} and Part (ii) is proved in Section \ref{sec:2}. We also have a "relative" version of Part (ii) of Theorem \ref{dense-one-under}, which we will prove in Section \ref{sec:4}.

\begin{rem}
Theorem \ref{dense-one-under} is also of interest when one considers the phenomenon of under-independence in mixing systems. While, as was mentioned above, mixing systems do not always have UI sets, they always have, so to say, almost UI sets.	
\end{rem}

In principle, it is conceivable that any weakly mixing system has an OI set, but we were not able to establish this. The following question is open.

\begin{ques}
	Let $(X,\mathcal{B},\mu,T)$ be an invertible probability measure preserving weakly mixing  system. Is there a set $A\in\mathcal{B}$ such that 
	$$\mu(A\cap T^{n}A)>\mu(A)^{2}$$
	for all $n\in\mathbb{Z}$?
\end{ques}

\subsection{Under- and over-independence for ergodic systems}	
We say that a system $(X,\mathcal{B},\mu,T)$ is \emph{free} (or the action $T$ is \emph{free}) if $T^{n}\neq id$ for all $n\neq 0$.
It is also natural to study modified versions of under- and over-independence for free ergodic systems. Taking into account the natural mode of convergence to independence in ergodic systems (see \ref{intro:3} above), we have the following analogue of Question \ref{q50}:
\begin{ques}\label{ques:7}
	Let $(X,\mathcal{B},\mu,T)$ be a free, invertible, ergodic probability measure preserving system and $M\in\mathbb{N}$.
	
	(i) Is there a set $A\in\mathcal{B}$ with $\mu(A)>0$ such that $\frac{1}{N}\sum_{n=M}^{N+M-1}\mu(A\cap T^{n}A)<\mu(A)^{2}$ for all $N\in\mathbb{N}$?
	
	(ii) Is there a set $A\in\mathcal{B}$ with $\mu(A)>0$ such that $\frac{1}{N}\sum_{n=M}^{N+M-1}\mu(A\cap T^{n}A)>\mu(A)^{2}$ for all $N\in\mathbb{N}$?
\end{ques} 		
We remark that the assumption that  $(X,\mathcal{B},\mu,T)$ is free can not be dropped due to the following simple observation. Assume that  $T^{k}=id$ for some $k\in\mathbb{N}$. Let $A\in\mathcal{B}, 0<\mu(A)<1$. By ergodic theorem, we have
$$\frac{1}{kN}\sum_{n=M}^{kN+M-1}\mu(A\cap T^{n}A)=\frac{1}{k}\sum_{n=M}^{k+M-1}\mu(A\cap T^{n}A)=\mu(A)^{2}$$
for all $N\in\mathbb{N}$. This implies that there exist infinitely may $N>0$ such that
$$\frac{1}{N}\sum_{n=M}^{N+M-1}\mu(A\cap T^{n}A)=\mu(A)^{2},$$
and so the answer to either  part of Question \ref{ques:7} is negative for such a system.


	We show in this paper that the answer to (ii) is YES while the answer to (i) is NO if $M=0$. 
	\begin{thm}[Existence of Ces\`aro OI set]\label{thm:e0} Let $(X,\mathcal{B},\mu,T)$ be a free, invertible, ergodic probability measure preserving system and $M\in\mathbb{N}$. 
		Then
		there is $A\in\mathcal{B}$ such that 
		$$\frac{1}{N}\sum_{n=M}^{N+M-1}\mu(A\cap T^{n}A)>\mu(A)^{2}$$
		for all $N\in\mathbb{N}$.
	\end{thm}
\begin{prop}[Nonexistence of Ces\`aro UI set for $M=0$]\label{prop:e0}
	Let $(X,\mathcal{B},\mu,T)$ be an invertible probability  measure preserving system. For any measurable set $A$ satisfying $\mu(A) \in (0,1)$, 
	there exist infinitely many $N \in \N$ such that 
	\[
	\frac{1}{N} \sum_{n=0}^{N-1} \mu ( A \cap T^{n}A) > \mu(A)^2 . 
	\]
\end{prop}

We remark that Question \ref{ques:7} (i) for $M > 0$ remains open.
We prove Theorem \ref{thm:e0} in Section \ref{sec:2} and Theorem \ref{prop:e0} in Section \ref{sec:3}.

\subsection{Over-independence for mixing of higher orders}
%
The mentioned above prevalence of over-independence manifests itself in a variety of additional situations.
We say that an invertible probability measure preserving system $(X,\mathcal{B},\mu,T)$ is \emph{mixing of order $d+1$} if
    	for all $A\in\mathcal{B}$, 	all integer sequences $(c_{i,n})_{n\in\mathbb{Z}},1\leq i\leq d$ such that $\lim_{\vert n\vert\to\infty}\vert c_{i,n}\vert=\lim_{\vert n\vert\to\infty}\vert c_{i,n}-c_{j,n}\vert=\infty$ for all $1\leq i < j\leq d$, 
we have
    	\begin{equation}\nonumber
    	\begin{split}
    	\lim_{\vert n\vert\to\infty}\mu(A\cap T^{c_{1,n}}A\cap \dots\cap T^{c_{d,n}}A )=\mu(A)^{d+1}.
    	\end{split}
    	\end{equation}

    	Methods similar to those used in the proofs of Theorems \ref{dense-one-under} Part (ii) and \ref{thm:e0} allow us to establish the following theorem:
    	\begin{thm}\label{thm:hm} Let $(X,\mathcal{B},\mu,T)$ be an invertible order-$(d+1)$ mixing probability measure preserving system. Let $(c_{i,n})_{n\in\mathbb{Z}}, 1\leq i\leq d$ be integer sequences such that $\lim_{\vert n\vert\to\infty}\vert c_{i,n}\vert=\lim_{\vert n\vert\to\infty} \vert c_{j,n}-c_{i,n}\vert=\infty$ for all $1\leq i<j\leq d$. Then there is $A\in\mathcal{B}$ such that 
    		$$\mu(A\cap T^{c_{1,n}}A\cap\dots\cap T^{c_{d,n}}A)>\mu(A)^{d+1}$$
    		for all $n\in\mathbb{Z}$.
    	\end{thm}

The proof of this theorem is given in Section \ref{sec:2}.
    	
%

\subsection{Over- and under-independence for action of amenable groups}
The definitions of ergodicity, weak mixing and mixing given at the beginning of the introduction can be naturally extended to the setup of amenable group actions. We deal with amenable group actions in Section \ref{sec:am}, where we show that any mixing measure preserving action of an amenable group has an OI set, and also formulate results which are analogous to Theorems \ref{dense-one-under} Part (ii) and \ref{thm:e0}.

%
%
%
%
\subsection{Organization of the paper} We prove the over-independence results (i.e. Part (ii) of Theorem \ref{dense-one-under}, \ref{thm:e0} and \ref{thm:hm}) in Section \ref{sec:2}, and the under-independence results (i.e. Part (i) of Theorem \ref{dense-one-under} and Proposition \ref{prop:e0}) in Section \ref{sec:3}. In Section \ref{sec:4}, we present the analogue of Part (ii) of Theorem \ref{dense-one-under} for relatively weakly mixing extensions. 
Finally, we deal with amenable group actions in Section \ref{sec:am}.

\section{Existence of over-independence sets}\label{sec:2}

%

We prove Theorem \ref{dense-one-under} Part (ii), \ref{thm:e0} and \ref{thm:hm} in this section. 

%

\begin{lem}\label{roh2}
	Let $(X,\mathcal{B},\mu,T)$ be a free, invertible, ergodic  probability measure preserving system. For every $C\in\mathcal{B}, N\in\N,\e>0$, $0<a<1-\mu(C)$, and every $c_{i,n}\in\mathbb{Z}, 1\leq i\leq d, \vert n\vert\leq N$, there exists $A\in\mathcal{B}$ such that $\mu(A)=a, A\cap C=\emptyset,$ and
	$$\mu((C\cup A)\cap T^{c_{1,n}}A\cap\dots\cap T^{c_{d,n}}A)>(1-\e)\mu(A)$$
	for all $\vert n\vert<N$. 
\end{lem}
\begin{proof}

	Let $M\in\mathbb{N}$ be such that $$M>\max\{\frac{1}{\epsilon a},d\max_{1\leq i\leq d,\vert n\vert\leq N}\vert c_{i,n}\vert\}.$$ Let $B$ be the base of a Rohlin Tower of height $M^{2}$ such that 
	$$\mu(\bigcup_{i=0}^{M^{2}-1}T^{i}B)>1-\epsilon.$$ 
	Choose a subset $I\subseteq B_{1}$ such that the set
	$$A=\{T^{i}x\colon 0\leq i<M^{2}, x\in B, T^{i}x\notin C\}$$
	has the property $\mu(A)=a$ (this can be achieved since $X$ is ergodic and free and thus atomless). Obviously $A\cap C=\emptyset$. Moreover, for all $\vert n\vert<N$, we have that
	$$\mu((C\cup A)\cap T^{c_{1,n}}A\cap\dots\cap T^{c_{d,n}}A)>\mu(A)-\frac{\sum_{i=1}^{d}\vert c_{i,n}\vert}{M^{2}}>\mu(A)-\frac{1}{M}>(1-\e)\mu(A).$$
\end{proof}

\begin{proof}[Proof of Theorem \ref{thm:hm}]
	The proof is similar to that of Theorem 3.1 in \cite{Adams}.
	
	 Let $0<a_{i},\e_{i}<1, i\in\mathbb{N}$ to be chosen later. Since $T$ is order-$d$ mixing, for every $B\in\mathcal{B}$, we have
	 $$\lim_{\vert n\vert\to\infty}\vert\mu(B\cap T^{c_{1,n}}B\cap\dots\cap T^{c_{d,n}}B)-\mu(B)^{d+1}\vert=0.$$
	 Let $A_{1}$ be an arbitrary set with $\mu(A_{1})=a_{1}$. 
	 There exists $N_{1}\in\N$ such that 
	 $$\vert\mu(A_{1}\cap T^{c_{1,n}}A_{1}\cap\dots\cap T^{c_{d,n}}A_{1})-\mu(A_{1})^{d+1}\vert<\e_{1}\mu(A_{1})^{d+1}$$
	 for all $\vert n\vert>N_{1}$.
	 
	 Suppose $A_{i}, N_{i}$ are chosen for all $i\leq k$. Denote $C_{j}=\bigcup_{i=1}^{j}A_{i}$. Let $A_{k+1}$ be
	 such that $\mu(A_{k+1})=a_{k+1}, A_{k+1}\cap C_{k}=\emptyset,$ and
	 $$\mu((C_{k}\cup A_{k+1})\cap T^{c_{1,n}}A_{k+1}\cap\dots\cap T^{c_{d,n}}A_{k+1})>(1-\e_{k})\mu(A_{k+1})$$
	 for all $\vert n\vert<N_{k}$. Since every mixing system is free and ergodic, the existence of $A_{k+1}$ is guaranteed by Lemma \ref{roh2} if $0<\sum_{i=1}^{\infty}a_{i}<1$. Let $N_{k+1}>N_{k}$ be such that 
	 $$\vert\mu(C_{k+1}\cap T^{c_{1,n}}C_{k+1}\cap\dots\cap T^{c_{d,n}}C_{k+1})-\mu(C_{k+1})^{d+1}\vert<\e_{k+1}\mu(C_{k+1})^{d+1}$$
	 for all $\vert n\vert>N_{k+1}$.
	 
	 Set $a_{i}=\frac{a}{i(i+1)}$, with $a$ sufficiently small.
	 We claim that $A=\bigcup_{i=1}^{\infty}A_{i}$ satisfies the condition of the theorem. If $\vert n\vert<N_{1}$, then
	  \begin{equation}\nonumber
	  \begin{split}
	  &\quad\mu(A\cap T^{c_{1,n}}A\cap\dots\cap T^{c_{d,n}}A)
	  \geq \sum_{k=2}^{\infty}\mu(A\cap T^{c_{1,n}}A_{k}\cap\dots\cap T^{c_{d,n}}A_{k})
	  \\&\geq \sum_{k=2}^{\infty}(1-\epsilon_{1})\mu(A_{k})=(1-\epsilon_{1})a/2>a^{d+1}=\mu(A)^{d+1},
	  \end{split}
	  \end{equation}
	  provided that $a$ is sufficiently small and $\e_{1}<1/2$.
	  Now suppose that $N_{k}\leq \vert n\vert<N_{k+1}$ for some $k\geq 1$. Then
	 \begin{equation}\nonumber
	 	\begin{split}
	 		&\quad\mu(A\cap T^{c_{1,n}}A\cap\dots\cap T^{c_{d,n}}A)
	 		\\&\geq \mu(C_{k}\cap T^{c_{1,n}}C_{k}\cap\dots\cap T^{c_{d,n}}C_{k})+\sum_{i=2}^{\infty}\mu(C_{k+i}\cap T^{c_{1,n}}A_{k+i}\cap\dots\cap T^{c_{d,n}}A_{k+i})
	 		\\&>(1-\e_{k})\mu(C_{k})^{d+1}+\sum_{i=2}^{\infty}(1-\e_{k+i})\mu(A_{k+i})
	 		\\&=(1-\e_{k})(a_{1}+\dots+a_{k})^{d+1}+\sum_{i=2}^{\infty}(1-\e_{k+i})a_{k+i}.
	 	\end{split}
	 \end{equation}
	 If $\e_{i}$ decreasing to 0 sufficiently fast, then 
	 \begin{equation}\nonumber
	 	\begin{split}
	 		&\quad(1-\e_{k})(a_{1}+\dots+a_{k})^{d+1}+\sum_{i=2}^{\infty}(1-\e_{k+i})a_{k+i}
	 		\\&>(1-\e_{k})((a-\frac{a}{k+1})^{d+1}+\frac{a}{k+2})>a^{d+1}=\mu(A)^{d+1}.
	 	\end{split}
	 \end{equation}  
\end{proof}

As an immediate corollary, we have:

\begin{cor} Let $(X,\mathcal{B},\mu,T)$ be an invertible order-$d$ mixing probability measure preserving system. Then there is $A\in\mathcal{B}$ such that 
	$$\mu(A\cap T^{n}A\cap\dots\cap T^{dn}A)>\mu(A)^{d+1}$$
	for all $n\in\mathbb{Z}$.
\end{cor}


The following lemma is straightforward:
\begin{lem}\label{lem:wmbh}
	Let $(X,\mathcal{B},\mu,T)$ be an invertible weakly mixing probability measure preserving system. Let $p_{1},p_{2},\dots,p_{d}$ be non-constant integer-valued polynomials such that $p_{i}-p_{j}$ are also non-constant for all $i\neq j$. Given a measurable set $C$ and $\epsilon > 0$, 
	there exists $N \in \N$ such that for any $n \geq N$ and measurable set $B$, 
	\[
	| \{ i\in \Z : |  i | \leq n, | \mu(B \cap T^{p_{1}(i)} C\cap \dots\cap T^{p_{d}(i)} C) - \mu(B) \mu(C)^{d} | \geq \epsilon \} | < \epsilon n .
	\]
\end{lem}
\begin{proof}
		Since $\Vert \bold{1}_{B}\Vert_{L^{2}(\mu)}\leq 1$, this lemma is an immediate corollary of Theorem D of \cite{BL}.  
\end{proof}

\begin{proof}[Proof of Theorem \ref{dense-one-under} Part (ii)]
	Let $0<a_{i},\e_{i}<1, i\in\mathbb{N}$ to be chosen later with $0<\sum_{i=1}^{\infty}a_{i}<1$. Since $T$ is weakly mixing and all of $p_{i}$, $p_{i}-p_{j}$ are also non-constant for all $i\neq j$, by \cite{BPet}, for every $B\in\mathcal{B}$, we have 
	$$\lim_{N\to\infty}\frac{1}{N}\sum_{n=0}^{N-1}\vert\mu(B\cap T^{p_{1}(n)}B\cap\dots\cap T^{p_{d}(n)}B)-\mu(B)^{d+1}\vert=0.$$
	So for every $\e>0$, there are infinitely many $N\in\N$ such that the set 
	$$\vert\{n\leq N\colon\vert\mu(B\cap T^{p_{1}(n)}B\cap\dots\cap T^{p_{d}(n)}B)-\mu(B)^{d+1}\vert>\e\mu(B)^{d+1}\}\vert<\e N.$$

	Let $A_{1}$ be an arbitrary set with $\mu(A_{1})=a_{1}$. Let $N_{1}\in\N$ be such that the cardinality of
	$$E_{1,N}:=\{n\leq N\colon\vert\mu(A_{1}\cap T^{p_{1}(n)}A_{1}\cap\dots\cap T^{p_{d}(n)}A_{1})-\mu(A_{1})^{d+1}\vert>\e_{1}\mu(A_{1})^{d+1}\}$$ is at most $\e_{1}N$ for all $N>N_{1}$.
	
	Suppose $A_{i}, N_{i}$ are chosen for all $i\leq k$. Denote $C_{j}=\bigcup_{i=1}^{j}A_{i}$. Since every weakly mixing system is ergodic and free, by Lemma \ref{roh2}, there exists a set $A_{k+1}$ with $\mu(A_{k+1})=a_{k+1}$, $A_{k+1}\cap C_{k}=\emptyset$ and 
	$$\mu((C_{k}\cup A_{k+1})\cap T^{p_{1}(n)}A_{k+1}\cap\dots\cap T^{p_{d}(n)}A_{k+1})>(1-\epsilon_{k})\mu(A_{k+1})$$
	for all $n\leq N_{k}$.
        For convenience, let $p_0(n)=0$ for all $n$. 
	Let $N_{k+1}>N_{k}$ be such that the cardinality of
	\begin{equation}
        E_{k+1,N} := \{n\leq N_{k+1}\colon\vert\mu(\bigcap_{i=0}^{d}T^{p_{i}(n)}C_{k+1})-\mu(C_{k+1})^{d+1}\vert 
        >\e_{k+1}\mu(C_{k+1})^{d+1}\}
        \end{equation} 
        is at most $\e_{k+1}N$ for all $N>N_{k+1}$.
	
	We claim that $A=\bigcup_{i=1}^{\infty}A_{i}$ satisfies the condition of the theorem. Suppose that $N_{k}\leq N<N_{k+1}$. If $n\notin E_{k,N}$, then
	\begin{equation}\nonumber
	\begin{split}
	&\quad\mu(A\cap T^{p_{1}(n)}A\cap\dots\cap T^{p_{d}(n)}A)
	\\&\geq \mu(C_{k}\cap T^{p_{1}(n)}C_{k}\cap\dots\cap T^{p_{d}(n)}C_{k})+\sum_{i=2}^{\infty}\mu(C_{k+i}\cap T^{p_{1}(n)}A_{k+i}\cap\dots\cap T^{p_{d}(n)}A_{k+i})
	\\&>(1-\e_{k})\mu(C_{k})^{d+1}+\sum_{i=2}^{\infty}(1-\e_{k+i})\mu(A_{k+i})
	\\&=(1-\e_{k})(a_{1}+\dots+a_{k})^{d+1}+\sum_{i=2}^{\infty}(1-\e_{k+i})a_{k+i}.
	\end{split}
	\end{equation}
	If we pick $a_{i}=\frac{a}{i(i+1)}$, $a$ sufficiently small, and $\e_{i}$ decreasing to 0 sufficiently fast, then 
	\begin{equation}\nonumber
	\begin{split}
	&\quad(1-\e_{k})(a_{1}+\dots+a_{k})^{d+1}+\sum_{i=2}^{\infty}(1-\e_{k+i})a_{k+i}
	\\&>(1-\e_{k})((a-\frac{a}{k+1})^{d+1}+\frac{a}{k+2})>a^{d+1}.
	\end{split}
	\end{equation}
	Since $\vert E_{k,N}\vert<\e_{k}N$ and $\e_{k}\to 0$, the set
	$$\{n\in\mathbb{N}\colon\mu(A\cap T^{p_{1}(n)}A\cap\dots\cap T^{p_{d}(n)}A)>\mu(A)^{d+1}\}$$
	is of density 1.   
\end{proof}

We now  prove the following theorem which is a more general form of Theorem \ref{thm:e0}: 
\begin{thm}[Ces\`aro over-independence]\label{thm:e00} Let $(X,\mathcal{B},\mu,T)$ be a free, invertible, ergodic probability measure preserving system and $M\in\mathbb{N}$. Then
	there is $A\in\mathcal{B}$ such that for all $k\in\mathbb{N}$, there exists $L_{k}\in\N$ such that
	$$\frac{1}{N}\sum_{n=M}^{N+M-1}\mu(A\cap T^{kn}A)>\mu(A)^{2}$$
	for all $N\geq L_{k}$.\footnote{The condition $N>L_{k}$ is necessary unless $\mu(A\cap T^{n}A)>\mu(A)^{2}$ for all $n\in\N$ (by Part (i) of Theorem \ref{dense-one-under} below, such a set does not always exist). To see this, suppose that $\mu(A\cap T^{n}A)\leq \mu(A)^{2}$ for all $n\in\N$. Then for $N=M=1$ and $k=n$, we have
		$$\frac{1}{N}\sum_{n=M}^{N+M-1}\mu(A\cap T^{kn}A)=\mu(A\cap T^{n}A)\leq \mu(A)^{2}.$$}
	Moreover, we can further require that $L_{k}=0$ for finitely many $k\in\N$.
\end{thm}

\begin{proof}
	Let $k_{0}\in\mathbb{N}$ be arbitrary and we will require that $L_{k}=0$ for all $k\leq k_{0}$ in the proof.
	Let $0<a_{i},\e_{i}<1, i\in\mathbb{N}$ to be chosen later. Let $I(T^{k})$ denote the $T^{k}$-invariant $\sigma$-algebra of $X$. 
        By the ergodic theorem, for every $B\in\mathcal{B}$, we have
	$$\lim_{N\to\infty}\frac{1}{N}\sum_{n=M}^{N+M-1}\mu(B\cap T^{kn}B)=\int_{X}\mathbb{E}(\bold{1}_{B}\vert I(T^{k}))^{2}\,d\mu,$$
	which in turn, implies that
	$$\lim_{N\to\infty}\frac{1}{N}\sum_{n=M}^{N+M-1}\mu(B\cap T^{kn}B)\geq \mu(B)^{2}.$$
	
	Let $A_{1}$ be an arbitrary set with $\mu(A_{1})=a_{1}$. 
	There exists $N_{1}\in\N$ such that 
	$$\frac{1}{N}\sum_{n=M}^{N+M-1}\mu(A_{1}\cap T^{n}A_{1})>(1-\e_{1})\mu(A_{1})^{2}$$
	for all $n>N_{1}$.
	
	Suppose $A_{i}, N_{i}$ are chosen for all $i\leq k$. Denote $C_{j}=\bigcup_{i=1}^{j}A_{i}$. Let $A_{k+1}$ be
	such that $\mu(A_{k+1})=a_{k+1}, A_{k+1}\cap C_{k}=\emptyset,$ and
	$$\mu((C_{k}\cup A_{k+1})\cap T^{n}A_{k+1})>(1-\e_{k})\mu(A_{k+1})$$
	for all $n<(k+k_{0})(N_{k}+\vert M\vert)$. The existence of $A_{k+1}$ is guaranteed by Lemma \ref{roh2} if $0<\sum_{i=1}^{\infty}a_{i}<1$. Let $N_{k+1}>(k+k_{0})(N_{k}+\vert M\vert)$ be such that 
	$$\frac{1}{N}\sum_{n=M}^{N+M-1}\mu(C_{k+1}\cap T^{mn}C_{k+1})>(1-\e_{k+1})\mu(C_{k+1})^{2}$$
	for all $n>N_{k+1}$ and $1\leq m\leq k+1$.
	
	Let $A=\bigcup_{i=1}^{\infty}A_{i}, L_{k}=N_{k}$ if $k>k_{0}$ and $L_{k}=0$ if $k\leq k_{0}$. We claim that such $A$ and $L_{k}$  satisfy the condition of the theorem. Fix $m\in\N$. We first assume that $m>k_{0}$. Let $N>L_{m}=N_{m}$ and suppose that $N_{k}\leq N<N_{k+1}$ for some $k\geq m$. Then
	\begin{equation}\nonumber
	\begin{split}
	&\quad\frac{1}{N}\sum_{n=M}^{N+M-1}\mu(A\cap T^{mn}A)
	\\&\geq \frac{1}{N}\sum_{n=M}^{N+M-1}\mu(C_{k}\cap T^{mn}C_{k})+\frac{1}{N}\sum_{n=M}^{N+M-1}\sum_{i=2}^{\infty}\mu(C_{k+i}\cap T^{mn}A_{k+i})
	\\&>(1-\e_{k})\mu(C_{k})^{2}+\sum_{i=2}^{\infty}(1-\e_{k+i})\mu(A_{k+i})
	\\&=(1-\e_{k})(a_{1}+\dots+a_{k})^{2}+\sum_{i=2}^{\infty}(1-\e_{k+i})a_{k+i}.
	\end{split}
	\end{equation}
	If we pick $a_{i}=\frac{a}{i(i+1)}$, $a$ sufficiently small, and $\e_{i}$ decreasing to 0 sufficiently fast, then 
	\begin{equation}\label{50}
	\begin{split}
	&\quad(1-\e_{k})(a_{1}+\dots+a_{k})^{2}+\sum_{i=2}^{\infty}(1-\e_{k+i})a_{k+i}
	\\&>(1-\e_{k})((a-\frac{a}{k+1})^{2}+\frac{a}{k+2})>a^{2}=\mu(A)^{2}.
	\end{split}
	\end{equation}  
	Now suppose $m\leq k_{0}$ and $N\geq L_{m}=0$. If $N<N_{1}$, then
	\begin{equation}\nonumber
	\begin{split}
	&\quad\frac{1}{N}\sum_{n=M}^{N+M-1}\mu(A\cap T^{mn}A)
	\geq \frac{1}{N}\sum_{n=M}^{N+M-1}\sum_{k=2}^{\infty}\mu(A\cap T^{mn}A_{k})
	\\&>\sum_{k=2}^{\infty}(1-\e_{1})\mu(A_{k})=(1-\e_{1})a/2>a^{2}=\mu(A)^{2}.
	\end{split}
	\end{equation}
	
	 If  $N_{k}\leq N<N_{k+1}$ for some $k\geq 1$, then 
		\begin{equation}\nonumber
		\begin{split}
		&\quad\frac{1}{N}\sum_{n=M}^{N+M-1}\mu(A\cap T^{mn}A)
		\\&\geq \frac{1}{N}\sum_{n=M}^{N+M-1}\mu(C_{k}\cap T^{mn}C_{k})+\frac{1}{N}\sum_{n=M}^{N+M-1}\sum_{i=2}^{\infty}\mu(C_{k+i}\cap T^{mn}A_{k+i})
		\\&>(1-\e_{k})\mu(C_{k})^{2}+\sum_{i=2}^{\infty}(1-\e_{k+i})\mu(A_{k+i})
		\\&=(1-\e_{k})(a_{1}+\dots+a_{k})^{2}+\sum_{i=2}^{\infty}(1-\e_{k+i})a_{k+i}.
		\end{split}
		\end{equation}
		The proof is finished by invoking (\ref{50}).
\end{proof}
The following proposition contrasts with the positive results on under- and over-independence 
by showing that ergodic translations on a compact group do not contain UI nor OI sets. 
\begin{prop}
	Let $X$ be a compact group with the normalized Haar measure $\mu$ and the $\sigma$-algebra  of the Borel set $\mathcal{B}$. Let $T$ be an ergodic translation on $X$. Then the measure preserving system 
	$(X,\mathcal{B},\mu,T)$ does not contain non-trivial UI or OI sets.
\end{prop}
\begin{proof}
	Note that any translation on a compact group is rigid, meaning that there exist a sequence of integers $(n_{i})_{i\in\mathbb{N}}$ such that for all $A\in\mathcal{B}$,
	\begin{equation}\label{4}
	\begin{split}
\Vert T^{-n_{i}}\bold{1}_{A}-\bold{1}_{A}\Vert_{L^{2}(\mu)}\to 0.
	\end{split}
	\end{equation}
	It follows from (\ref{4}) that for all $A\in\mathcal{B}$, $\mu(A\cap T^{n_{i}}A)\to \mu(A)$ as $i\to\infty$, which clearly implies that 	$(X,\mathcal{B},\mu,T)$ contains no non-trivial UI sets.

	Note that for any $A\in\mathcal{B}$ and $\e>0$,  there exists a syndetic set $E\subset\mathbb{N}$ such that 
	$\vert\mu(A\cap T^{n}A)-\mu(A)\vert<\epsilon$ for all $n\in E$.
	Now suppose that $A\in\mathcal{B}, 0<\mu(A)<1,$ is an OI set. 
	Then
	\begin{equation}\nonumber
		\begin{split}
			\qquad&\lim_{N\to\infty}\frac{1}{N}\sum_{n=0}^{N-1} \mu(A\cap T^{n}A)
			\\&=\lim_{N\to\infty}\frac{1}{N}\Bigl(\sum_{0\leq n< N, n\in E} \mu(A\cap T^{n}A)+\sum_{0\leq n< N, n\notin E} \mu(A\cap T^{n}A)\Bigr)
			\\&\geq \lim_{N\to\infty}\frac{1}{N}\Bigl(\sum_{0\leq n< N, n\in E} (\mu(A)-\epsilon)+\sum_{0\leq n< N, n\notin E} \mu(A)^{2}\Bigr)
			\\&=d^{*}(E)(\mu(A)-\epsilon)+(1-d^{*}(E))\mu(A)^{2}.
		\end{split}
	\end{equation}
	Since $\epsilon$ is arbitrary and $d^{*}(E)>0$, we have
	$$\lim_{N\to\infty}\frac{1}{N}\sum_{n=0}^{N-1} \mu(A\cap T^{n}A)>\mu(A)^{2},$$
	which contradicts the ergodic theorem. So 	$(X,\mathcal{B},\mu,T)$ does not contain OI sets.
\end{proof}

%

\section{Positive and negative results for under-independence sets}\label{sec:3}

In this section we prove Part (i) of Theorem \ref{dense-one-under} 
and Proposition \ref{prop:e0}. 

\subsection{Proof of Part (i) of Theorem \ref{dense-one-under}}
We start with a general procedure for constructing the candidate set $A$.  A sequence 
of parameters is used to construct $A$. 
Then we show how to choose the parameters 
such that $A$ is a density-1 UI set. 

\subsubsection{Set Engineering}
Let $d\in\mathbb{N}$ and polynomials $p_{1},\dots,p_{d}$ be fixed. We may assume without loss of generality that when $n>0$, all $p_{i}(n)$ are monotone increasing and $0<p_{1}(n)<p_{2}(n)<\dots<p_{d}(n)$.
Let $q$ be any prime number such that for all $j\in\N$, we have that 
$$\Bigl\vert\{0\leq n<q^{j}\colon q^{j}\vert p_{1}(n)\}\Bigr\vert\leq \deg(p_{1})$$
(this can be achieved by picking $q$ such that  $p_{1}(x)\not\equiv 0\mod q$ as an element in $(\mathbb{Z}/q\mathbb{Z})[x]$, then $p_{1}$ has at most $\deg(p_{1})$ roots in  $(\mathbb{Z}/q^{j}\mathbb{Z})[x]$).

Denote $c=\frac{d+1}{d}$ and $S=\sum_{p=1}^{\infty}\frac{1}{p^{c}}$. Let $a \in (0,{S} / {100qd})$ be a real number. 
Define $a_p = {a} / { S p^{c}}$ for $p \geq 1$. 
Observe that 
\[ 
\sum_{p=1}^{\infty} a_p  =  a . 
\]
Let $\alpha \in \N$ be such that $\alpha > \frac{(d+1)d^{d+1}S^{d}}{a^{d}}$. 
For $i \geq 0$, 
define $c_{\alpha p+j} = {a_{p+1}} / {\alpha}$ for $0 \leq j \leq \alpha - 1$. 
Also, define $b_i$ such that 
$$ b_n = q^m c_n$$ 
for $\alpha ((q+1)^{dm} - 1) \leq n < \alpha ((q+1)^{d(m+1)} - 1)$. 
Note that 
$$\sum_{n=1}^{\infty}c_{n}=\sum_{p=1}^{\infty}a_{p}=a$$
and
\begin{align}
\sum_{n=1}^{\infty} b_n &= \sum_{m=0}^{\infty} \sum_{p=(q+1)^{dm}}^{(q+1)^{d(m+1)}-1} \frac{q^m a}{Sp^{c}} 
\leq \frac{a}{S} \sum_{m=0}^{\infty} q^m \int_{(q+1)^{dm}}^{(q+1)^{d(m+1)}} \frac{1}{x^{c}}\,dx \\ 
&= \frac{ad}{S} \sum_{m=0}^{\infty} q^m \frac{q}{(q+1)^{m+1}}  =  qda/S. 
\end{align}

Let $\ell_n = q^m=\frac{b_{n}}{c_{n}}$ if $\alpha ((q+1)^{dm} - 1) \leq n < \alpha ((q+1)^{d(m+1)} - 1)$ for some $m \geq 0$ and denote $m_{n}=\log_{q}\ell_{n}$. Let $(\epsilon_{n})_{n\in\N}$ be a sequence of positive numbers tending to 0 sufficiently fast.  $(\epsilon_{n})_{n\in\N}$ depends only on $(a_{n})_{n\in\N}$ and its choice will be clear in the proof.
 
\subsubsection{Construction of the density-1 UI set}
Let $(X,\mathcal{B},\mu,T)$ be a probability measure preserving system. For $L_{0}\in\mathbb{N}$ and $\e>0$, we say that a set $D\in\mathcal{B}$ is \emph{$(L_{0},\e)$-uniform} if for every Rohlin tower $\bigcup_{i=0}^{L}T^{i}B$ of height $L\geq L_{0}$ with $\mu(\bigcup_{i=0}^{L}T^{i}B)>1-\e$ and every $I\subseteq B$, we have that
$$\Bigl\vert\mu(\bigcup_{i=0}^{L}T^{i}I\cap D)-\mu(D)\mu(\bigcup_{i=0}^{L}T^{i}I)\Bigr\vert<\e\mu(D)\mu(\bigcup_{i=0}^{L}T^{i}I).$$

\begin{lem}\label{lem:uniform}
	Let $(X,\mathcal{B},\mu,T)$ be a probability measure preserving system. For all $D\in\mathcal{B}$ with $\mu(D)>0$ and $\e>0$, there exists $L_{0}\in\mathbb{N}$ such that $D$ is $(L_{0},\e)$-uniform.
\end{lem}	 
\begin{proof}
	Fix $D$ and $\e$.
	By ergodic theorem, there exists $L_{0}\in\mathbb{N}$ such that for all $L>L_{0}$, we have that $$\Bigl\Vert\frac{1}{L+1}\sum_{i=0}^{L}\bold{1}_{D}\circ T^{i}-\mu(D)\Bigr\Vert_{L^{2}(\mu)}<\e\mu(D).$$ So for all Rohlin tower $\bigcup_{i=0}^{L}T^{i}B$, we have that
	\begin{equation}\nonumber
	\begin{split}
	 &\quad \Bigl\vert\mu(\bigcup_{i=0}^{L}T^{i}I\cap D)-\mu(D)\mu(\bigcup_{i=0}^{L}T^{i}I)\Bigr\vert
	 \\&=\Bigl\vert\sum_{i=0}^{L}\int_{X}\bold{1}_{D}(x)\bold{1}_{I}(T^{-i}x)\,d\mu-(L+1)\mu(D)\mu(I)\Bigr\vert
	 \\&=\Bigl\vert\sum_{i=0}^{L}\int_{X}\bold{1}_{D}(T^{i}x)\bold{1}_{I}(x)\,d\mu-(L+1)\mu(D)\mu(I)\Bigr\vert
	 \leq(L+1)\e\mu(D)\mu(I),
	\end{split}
	\end{equation}
	which finishes the proof.
\end{proof}

We construct inductively a sequence of disjoint sets $A_{n}$ with $\mu(A_{n})=c_{n}$, and then show that the set $A=\cup_{n=1}^{\infty}A_{n}$ is what we want.

Let $h_{1}=1, r_{1}=k_{1}=0$. Let $A_{1}$ be an arbitrary set with $\mu(A_{1})=c_{1}$. Let $D_{1}=X\backslash A_{1}$, $E_{1}=F_{1}=A_{1}$ and $B_{1}=C_{1}=X$ (in fact the only useful information is that $\mu(A_{1})=c_{1}$, and all other parameters are just chosen for convenience). 

Denote  $\overline{A}_{n}=\cup_{i=1}^{n}A_{i}$ and
$d_n = \mu ( \bigcap_{i=1}^{n} D_i)$ (write $d_{0}=1$).
Suppose that the following have been defined:
$$h_{j}, r_{j}, k_{j}, A_{j}, B_{j}, 
C_{j}, D_{j}, E_{j}, F_{j}$$
for all $j<n$ for some $n\geq 2$ such that for all $j<n$, we have the following conditions: 
\begin{enumerate}
	\item {$\mu(A_{j})=c_{j}$ and $A_{1},\dots,A_{j}$ are pairwise disjoint;} \label{enum-3}
	\item {$C_{j}=\cup_{i=0}^{h_{j}-1}T^{i}B_{j}$ is a Rohlin tower of height $h_{j}$ and base $B_{j}$ such that $\mu(C_{j})>1-\e_{j}$ and $C_{j}$ is the disjoint union of $D_{j}$ and $F_{j}$;} \label{enum-6}
	\item {$\mu(E_{j})<2b_{j}\mu(B_{j})/d_{j-1}$, $\mu(F_{j})\leq 10b_{j}$ and $d_{j}>1-10qad/S$;}\label{enum-4}
	\item{if $j>1$, then for all measurable set $G$ and all $k>k_{j}$, 
		\[
		\Bigl\vert\{ i : 0\leq i < k, | \mu (H_{0}\cap T^{p_{1}(i)}H_{1}\cap\dots\cap T^{p_{d}(i)}H_{d}) - \mu(G)\mu (\bar{A}_{j-1})^{d} | 
		> \epsilon_{j}  \}\Bigr\vert < \epsilon_{j}  k 
		\] 
		whenever at most one of $H_{0},\dots,H_{d}$ equals to $G$ and all the others equals to $\bar{A}_{j-1}$;
	} \label{enum-1}
	\item{
		if $j>1$, then for all $I\subseteq \{1,2,\dots,j-1\}$, the set $\bigcap_{i\in I}D_{i}$ is  
		$(r_{j},\e_{j})$-uniform;} \label{enum-2}
	\item{$r_{j}>\max_{1\leq i\leq d}p_{i}(k_{j})$ if $j>1$.} \label{enum-5}
\end{enumerate}

It is easy to check that $h_{1}, r_{1}, k_{1}, A_{1}, B_{1}, 
C_{1}, D_{1}, E_{1}, F_{1}$  satisfy all the properties above.
 We now define inductively 
$$h_{n}, r_{n}, k_{n}, A_{n}, B_{n}, 
C_{n}, D_{n}, E_{n}, F_{n}$$
such that they satisfy the same properties above with $j$ replaced with $n$.

By Lemma \ref{lem:wmbh}, there exists $k_{n}\in\N$ such that
for all $k \geq k_{n}$, conditions (\ref{enum-1}) hold for $j=n$.
By Lemma \ref{lem:uniform}, we may pick $r_{n}>\max_{1\leq i\leq d}p_{i}(k_{n})$ such that condition (\ref{enum-2}) holds for $j=n$. Then condition (\ref{enum-5}) also holds for $j=n$.
Let $C_{n}$ be a Rohlin tower of height $h_{n}$ ($h_{n}\gg_{\ell_{n}}r_{n}$ to be chosen later) with base $B_{n}$ such that
\[
\mu ( \bigcup_{i=0}^{h_n - 1} T^i B_n ) > 1 - \epsilon_{n}. 
\]
For every $E\subseteq B_{n}$, denote 
$$R(E):=\left \{ T^{i \ell_{n}} x : r_{n} \leq i < h_{n} - r_{n}, 
x \in E, T^{i \ell_{n}} x \in \bigcap_{j=1}^{n-1} D_j \right \}.$$

Since $\bigcap_{j=1}^{n-1} D_j$ is $(r_{n},\e_{n})$-uniform, if $h_{n}$ is sufficiently large and $\e_{n}$ is sufficiently small (but $\e_{n}$ depends only on $a_{n}$), $\mu(R(E))$ is approximately $\frac{d_{n-1}\mu(E)}{\ell_{n}\mu(B_{n})}$. 
Since $\frac{d_{n-1}}{\ell_{n}}>\frac{1-3qa}{\ell_{n}}>\frac{b_{n}}{\ell_{n}}=c_{n}$,
since every weakly mixing system is ergodic and free and thus atomless, there exists $E_{n}\subseteq B_{n}$ such that the set
$$A_{n}:=\left \{ T^{i \ell_{n}} x : r_{n} \leq i < h_{n} - r_{n}, 
x \in E_{n}, T^{i \ell_{n}} x \in \bigcap_{j=1}^{n-1} D_j \right \}$$
is of measure $c_{n}$ and 
$$\mu(E_{n})<\frac{2c_{n}\ell_{n}\mu(B_{n})}{d_{n-1}}=\frac{2b_{n}\mu(B_{n})}{d_{n-1}},$$
provided that $h_{n}$ is sufficiently large and $\e_{n}$ is sufficiently small.
For this $A_{n}$, obviously condition (\ref{enum-3}) holds for $j=n$.

Now
Let $F_n = \cup_{i=0}^{h_n - 1} T^i E_n$ and 
$D_n = \cup_{i=0}^{h_n - 1} T^i (B_n \setminus E_n)$. Thus, $C_n = F_n \cup D_n$ and (\ref{enum-6}) is satisfied. Moreover, $\mu(F_{n})=h_{n}\mu(E_{n})<\frac{2b_{n}}{d_{n-1}}<\frac{2b_{n}}{1-10qad/S}<10b_{n}$, and 
$$d_{n}=\mu(\cap_{i=1}^{n}D_{i})\geq 1-\sum_{i=1}^{n}10b_{i}>1-10qad/S.$$
So condition (\ref{enum-4}) holds for $j=n$. This finishes the construction for $j=n$.

\subsubsection{End of the proof}
We now show that the set $A$ constructed in the previous section is what we want. We start with the following lemma:
\begin{lem}
	\label{lem12} Let the notations be as in the previous section. Then for all $i_{2}>i_{1}$ and $0\leq j<r_{i_{1}}$, we have that 
	$$\mu(T^{j}A_{i_{1}}\cap A_{i_{2}})=0.$$
\end{lem}
\begin{proof}
	Note that $A_{i_{2}}\subset D_{i_{1}}$ and $A_{i_{1}}\subset F_{i_{1}}$. Since $0\leq j<r_{i_{1}}$, by the choice of $A_{i_{1}}$, we have that $T^{j}A_{i_{1}}\subset F_{i_{1}}$ and so
	$$\mu(T^{j}A_{i_{1}}\cap A_{i_{2}})\leq \mu(T^{j}A_{i_{1}}\cap D_{i_{1}})\leq \mu(F_{i_{1}}\cap D_{i_{1}})=0.$$
\end{proof}

Given sufficiently large $k \in \N$, 
let $n$ be such that $k_{n}\leq k < k_{n+1}$. 
Also, choose $p$ such that 
$\alpha p \leq n < \alpha ( p + 1 )$. WLOG, assume $p > 0$. 
Denote $$U_{n}=A\backslash\overline{A_{n-1}}=\cup_{i=n}^{\infty}A_{i}.$$
By condition (\ref{enum-1}), there exists $W\subseteq \{0,1,\dots,k\}$ with $\vert W\vert>(1-\e_{n-1})k$  such that for all $i\in W$
\begin{equation}\nonumber
\begin{split}
\vert \mu(H_{0}\cap T^{p_{1}(i)}H_{1}\cap\dots\cap T^{p_{d}(i)}H_{d})-\mu(U_{n})\mu(\overline{A_{n-1}})^{d}\vert<\e_{n-1}
\end{split}
\end{equation}
whenever at most one of $H_{0},\dots,H_{d}$ equals to $U_{n}$ and all the others equals to $\bar{A}_{n-1}$. So for all $i\in W$, we have that 
\begin{equation}\label{equ:temp21}
\begin{split}
&\quad \mu(A\cap T^{p_{1}(i)}A\cap\dots\cap T^{p_{d}(i)}A)-\mu(A)^{d+1}
\\&< \mu(U_{n}\cap T^{p_{1}(i)}U_{n}\cap\dots\cap T^{p_{d}(i)}U_{n})-\mu(U_{n})^{d+1}+(d+1)\e_{n-1}.
\end{split}
\end{equation}
Since $p_{i}(k)<p_{i}(k_{n+1})<r_{n+1}$, by Lemma \ref{lem12}, we have that 
\begin{equation}\nonumber
\begin{split}
\mu(A_{j_{0}}\cap T^{p_{1}(i)}A_{j_{1}}\cap\dots\cap T^{p_{d}(i)}A_{j_{d}})=0
\end{split}
\end{equation}
if $j_{a}>j_{b}>n$ for some $0\leq a,b\leq d$. So
\begin{equation}\nonumber
\begin{split}
&\quad\mu(U_{n}\cap T^{p_{1}(i)}U_{n}\cap\dots\cap T^{p_{d}(i)}U_{n})
\\&\leq \sum_{j=n+1}^{\infty}\mu(A_{j}\cap T^{p_{1}(i)}A_{j}\cap\dots\cap T^{p_{d}(i)}A_{j})+(d+1)\mu(A_{n})
\\&\leq \sum_{j=n+1}^{\infty}\mu(A_{j}\cap T^{p_{1}(i)}A_{j})+(d+1)\mu(A_{n}).
\end{split}
\end{equation}
By the construction of $A_{j}$, 
$\mu(A_{j}\cap T^{p_{1}(i)}A_{j})=0$ unless $\ell_{j}=q^{m_{j}}$ divides $p_{1}(i)$. Let 
$W'=W\backslash \ell_{n+1}\mathbb{Z}$. Then by the choice of $q$, we have that
$$\vert W'\vert>(1-\e_{n-1}-\frac{\deg(p_{1})}{\ell_{n+1}})k$$
and for all $i\in W'$, since $\ell_{i}\vert\ell_{i+1}$, we have that 
\begin{equation}\label{equ:temp22}
\begin{split}
\quad\mu(U_{n}\cap T^{p_{1}(i)}U_{n}\cap\dots\cap T^{p_{d}(i)}U_{n})
\leq (d+1)\mu(A_{n}).
\end{split}
\end{equation}
Combining (\ref{equ:temp21}) and (\ref{equ:temp22}), we have that for all $i\in W'$,
\begin{equation}\nonumber
\begin{split}
&\quad\mu(A\cap T^{p_{1}(i)}U_{n}\cap\dots\cap T^{p_{d}(i)}A)-\mu(A)^{d+1}
\\&\leq (d+1)\mu(A_{n})+(d+1)\e_{n-1}-\mu(U_{n})^{d+1}
\\&\leq (d+1)a_{p+1}/\alpha+(d+1)\e_{n-1}-(\sum_{n=p+2}^{\infty}a_{p})^{d+1}.
\end{split}
\end{equation}
Since
\begin{equation}\nonumber
\begin{split}
&\quad(\sum_{n=p+2}^{\infty}a_{p})^{d+1}=(\sum_{n=p+2}^{\infty}\frac{a}{Sn^{c}})^{d+1}
\geq (\int_{p+1}^{\infty}\frac{a}{Sx^{c}}\,dx)^{d+1}=(\frac{da}{S(p+1)^{\frac{1}{d}}})^{d+1}
\\&=(da/S)^{d+1}\cdot (p+1)^{-\frac{d+1}{d}}>\frac{(d+1)a}{S\alpha}\cdot (p+1)^{-\frac{d+1}{d}}=(d+1)a_{p+1}/\alpha,
\end{split}
\end{equation}
and $\vert W'\vert/k\to 1$ as $k\to\infty$ (and so $n\to\infty$),
we have that the set of $i\in\mathbb{N}$ with
$$\mu(A\cap T^{p_{1}(i)}U_{n}\cap\dots\cap T^{p_{d}(i)}A)<\mu(A)^{d+1}$$
is of density 1 provided $\e_{n}\to 0$ sufficiently fast (since $\alpha p \leq n < \alpha ( p + 1 )$, $\e_{n}$ can be chosen to be dependent only on $a$ and $d$). This finishes the proof of Part (i) of Theorem \ref{dense-one-under}.

\subsection{Proof of Proposition \ref{prop:e0}}

We now prove Proposition \ref{prop:e0}.

\begin{proof}[Proof of Proposition \ref{prop:e0}]
	Assume the proposition is not true.  
	Then
	there exist $\e>0$, $A\in \mathcal{B}$ and $n \in \N$ such that for $m > n$, 
	\[
	\frac{1}{m} \sum_{i=0}^{m-1} \mu ( A \cap T^i A ) \leq \mu(A)^2 . 
	\]	
	By the Cauchy-Schwarz inequality, for any $N \in \N$, 
	\begin{align}
	\int_{X} \big( \frac{1}{N} \sum_{i=0}^{N-1} I_{T^i A}(x) \big)^2 d\mu &\geq \mu(A)^2 \label{eqn:cs}. 
	\end{align}
	The left-handside of (\ref{eqn:cs}) may be decomposed into the following four parts:
	\begin{align}
	\frac{1}{N^2} \sum_{i = 0}^{N-1} \sum_{j = 0}^{N-1} \mu (T^i A \cap T^j A) &= 
	\frac{1}{N^2} \sum_{i=0}^{n-1} \sum_{j=0}^{i} \mu(A\cap T^j A) \label{sq1} \\
	&+ \frac{1}{N^2} \sum_{i=n}^{N-1} \sum_{j=0}^{i} \mu(A\cap T^j A) \label{sq2} \\
	&+ \frac{1}{N^2} \sum_{i=1}^{n-1} \sum_{j=1}^{i} \mu(A\cap T^j A) \label{sq3} \\
	&+ \frac{1}{N^2} \sum_{i=n}^{N-1} \sum_{j=1}^{i} \mu(A\cap T^j A) \label{sq4} 
	\end{align}
	Note that quantity (\ref{sq4}) above satisfies:
	\begin{align}\nonumber
	\frac{1}{N^2} \sum_{i=n}^{N-1} \sum_{j=1}^{i} \mu(A\cap T^j A) &= 
	\frac{1}{N^2} \sum_{i=n}^{N-1} \sum_{j=0}^{i} \mu(A\cap T^j A) - \frac{1}{N^2} \sum_{i=n}^{N-1} \mu(A) \\ 
	&\nonumber\leq \frac{1}{N^2} \sum_{i=n}^{N-1} (i+1) \mu(A)^2 - (\frac{N-n}{N^2}) \mu(A)  \\ 
	&\nonumber= \frac{1}{N^2} \big( \frac{N(N+1)}{2} - \frac{n(n+1)}{2} \big) \mu(A)^2 - (\frac{N-n}{N^2}) \mu(A). 
	\end{align}
	Also, terms (\ref{sq1}) and (\ref{sq3}) satisfy:
	\begin{align}\nonumber
	(\ref{sq1}) + (\ref{sq3}) &\leq (\frac{n}{N})^2 \mu(A) . 
	\end{align}
	Term (\ref{sq2}) satisfies:
	\begin{align}\nonumber
	\frac{1}{N^2} \sum_{i=n}^{N-1} \sum_{j=0}^{i} \mu(A\cap T^j A) &\leq 
	\frac{1}{N^2} \sum_{i=n}^{N-1} (i + 1) \mu(A)^2 \\ 
	&\nonumber= \frac{1}{N^2} \big( \frac{N(N+1)}{2} - \frac{n(n+1)}{2} \big) \mu(A)^2 . 
	\end{align}
	Adding up the terms:
	\begin{align}\nonumber
	(\ref{sq1}) + (\ref{sq2}) + (\ref{sq3}) + (\ref{sq4}) &\leq 
	\frac{1}{N^2} \big( N(N+1) - n(n+1) \big) \mu(A)^2 + (\frac{n}{N})^2 \mu(A) \\ 
	&\nonumber- (\frac{N-n}{N^2}) \mu(A) \\ 
	&= \big( 1 + \frac{1}{N} - (\frac{n}{N})^2 - \frac{n}{N^2} \big) \mu(A)^2 
	- ( \frac{N - n - n^2}{N^2} ) \mu(A) \label{sqall}. 
	\end{align}
	However, for sufficiently large $N$, the value of (\ref{sqall}) is less than $\mu(A)^2$. 
	This contradicts (\ref{eqn:cs}). 
\end{proof}


\section{Relative over-independence for weakly mixing extensions}\label{sec:4}



In this section, in order to be safe when dealing with some  measure-theoretical constructions, we will be assuming that measure spaces are \emph{regular}, that is, isomorphic to spaces of the form $(X,\mathcal{B},\mu)$, where $X$ is a compact metric space, $\mathcal{B}$ is the $\sigma$-algebra of Borel sets and $\mu$ is a probability measure on $X$.

Let $\pi\colon(X,\mathcal{B},\mu,T)\to (Y,\mathcal{C},\nu,T)$ be the factor map between two invertible probability measure preserving systems. For $\tau>0$, we say a set $C\in\mathcal{B}$ is a \emph{$\tau$-regular} if $\E(\bold{1}_{C}\vert \mathcal{C})(y)$ equals to either 1 or $\tau$ for $\nu$-a.e. $y\in Y$.

We say that $(X,\mathcal{B},\mu,T)$ is a \emph{relatively weakly mixing extension} of $(Y,\mathcal{C},\nu,T)$ if $X$ is an extension of $Y$ and for all $f,g\in L^{\infty}(\mu)$, we have that 
$$\lim_{N\to\infty}\frac{1}{N}\sum_{n=0}^{N-1}\int_{X}\Bigl\vert\mathbb{E}(f\cdot T^{n}g\vert\mathcal{C})-\mathbb{E}(f\vert\mathcal{C})T^{n}\mathbb{E}(g\vert\mathcal{C})\Bigr\vert\,d\mu=0,$$


We have the following "over-independence" result for the relative case:
\begin{prop}\label{thm:rwm} 
	Let $(Y,\mathcal{C},\nu,T)$ be a free, invertible, ergodic probability measure preserving system, and let $(X,\mathcal{B},\mu,T)$ be a nontrivial relatively weakly mixing extension of $(Y,\mathcal{C},\nu,T)$.
	Let $p_{1},p_{2},\dots,p_{d}$ be non-constant integer-valued polynomials such that $p_{i}-p_{j}\not\equiv const$ for all $i\neq j$.
	Then there exists $A\in\mathcal{B}$ such that for $f=\bold{1}_{A}$, the set
	\begin{equation}\label{temp50}
	\begin{split}
	\Bigl\{n\in\mathbb{N}\colon\int_{X} \prod_{i=0}^{d} T^{p_{i}(n)}f\, d\mu>\int_{X}\prod_{i=0}^{d}  \E(T^{p_{i}(n)}f\vert\mathcal{C})\, d\mu\Bigr\}
	\end{split}
	\end{equation}
	is of density 1. In particular, for any $\tau>0$ such that $(d+1)\tau^{d}<1$,  $A$ can be chosen to be a $\tau$-regular set.
\end{prop}

In order to prove Proposition \ref{thm:rwm}, we need first some lemmas.

\begin{lem}\label{lem:00}
		Let $(Y,\mathcal{C},\nu,T)$ be a free, invertible, ergodic probability measure preserving system, and let $(X,\mathcal{B},\mu,T)$ be a nontrivial relatively weakly mixing extension of $(Y,\mathcal{C},\nu,T)$.  Let $\mu=\int_{Y}\mu_{y}\,d\nu(y)$ be the disintegration of $\mu$ with respect to $\nu$. Then $\mu_{y}$ is atomless for $\nu$-a.e. $y\in Y$. 
\end{lem}
\begin{proof}
	For $\nu$-a.e. $y\in Y$, there is a unique way to write $\mu_{y}=\mu_{y,c}+\mu_{y,d}$, where $\mu_{y,c}$ is an atomless measure and $\mu_{y,d}$ is an atomic measure (meaning that $\mu_{y,d}$ is supported on at most countable many atoms). Let $\mu_{c}=\int_{Y}\mu_{y,c}\,d\nu(y)$ and $\mu_{d}=\int_{Y}\mu_{y,d}\,d\nu(y)$. Since for all $A\in\mathcal{B}$, the map $y\to \mu_{y}(A)$ is measurable with respect to $\mathcal{C}$,  the maps $y\to \mu_{y,c}(A)$ and $y\to \mu_{y,d}(A)$ are also measurable with respect to $\mathcal{C}$ (see Theorems 2.1 and 2.12 in \cite{DF}). This implies that every $A\in\mathcal{B}$ is both $\mu_{c}$- and $\mu_{d}$-measurable (we caution the reader that $\mu_{c}$ and $\mu_{d}$ are not normalized, i.e. $\mu_{c}(X)$ and $\mu_{d}(X)$ may be not equal to 1). 
	
	\
	
	We claim that both $\mu_{c}$ and $\mu_{d}$ are $T$-invariant.
	Let $(T)_{\ast}$ denote the pushforward of measures under $T$. Since $\nu$ is $T$-invariant, $(T)_{\ast}\mu_{c}=\int_{Y}(T)_{\ast}\mu_{y,c}\,d\nu(y)=\int_{Y}(T)_{\ast}\mu_{T^{-1}y,c}\,d\nu(y):=\int_{Y}\mu'_{y,c}\,d\nu(y)$, where $\mu'_{y,c}:=(T)_{\ast}\mu_{T^{-1}y,c}$ is a measure supported on $\pi^{-1}(y)$. Since the pushforward of $T$ maps any atomless measure on $\mathcal{B}$ to an atomless measure, $\mu'_{y,c}$ is atomless. Similarly, since the pushforward of $T$ maps any atomic measure on $\mathcal{B}$ to an atomic measure, we have that $(T)_{\ast}\mu_{d}=\int_{Y}\mu'_{y,d}\,d\nu(y)$, where $\mu'_{y,d}:=(T)_{\ast}\mu_{T^{-1}y,d}$ is an atomic measure supported on $\pi^{-1}(y)$.

	Since $\mu$ is $T$-invariant, we have that $(T)_{\ast}\mu_{c}(A)+(T)_{\ast}\mu_{d}(A)=(T)_{\ast}\mu(A)=\mu(A)=\mu_{c}(A)+\mu_{d}(A)$ for all $A\in\mathcal{B}$. This implies that $\mu_{y,c}+\mu_{y,d}=\mu'_{y,c}+\mu'_{y,d}$ for $\nu$-a.e. $y\in Y$. By the uniqueness of the decomposition, we have that $\mu_{y,c}=\mu'_{y,c}=(T)_{\ast}\mu_{T^{-1}y,c}$, which implies that $$(T)_{\ast}\mu_{c}=\int_{Y}(T)_{\ast}\mu_{y,c}\,d\nu(y)=\int_{Y}\mu_{Ty,c}\,d\nu(y)=\int_{Y}\mu_{y,c}\,d\nu(y)=\mu_{y,c}.$$
	Similarly, $\mu_{d}$ is also a $T$-invariant measure. This proves the claim.
	
	\
	
	Since $(Y,\mathcal{C},\nu,T)$ is ergodic and $(X,\mathcal{B},\mu,T)$ is a weakly mixing extension of $(Y,\mathcal{C},\nu,T)$, $(X,\mathcal{B},\mu,T)$ is also ergodic.
	Since $\mu=\mu_{c}+\mu_{d}$ and all of the three measures are $T$-invariant and $\mu$ is ergodic, we have that $\mu_{c}=k\mu$ and $\mu_{d}=(1-k)\mu$ for some $0\leq k\leq 1$.
%
	
	If $k\neq 0$, then $\mu_{y}=k^{-1}\mu_{y,c}$ is atomless for $\nu$-a.e. $y\in Y$ and we are done.

	Now we assume that $k=0$ and so
	$\mu_{d}=\mu$. Since $\mu_{y}=\mu_{y,d}$ is atomic for $\nu$-a.e. $y\in Y$ (as all the spaces considered in this paper are standard), every point in $\pi^{-1}(y)$ is an atom for $\mu_{y}$ for $\nu$-a.e. $y\in Y$. By the Measurable Choice Theorem (see,  for example, \cite{Au}), for $\nu$-a.e. $y\in Y$, there exists an atom $x_{y}\in X$ such that $\pi(x_{y})=y$ and the set $C:=\{x_{y}\in X\colon y\in Y\}$ is a measurable set.  Let $f=\bold{1}_{C}$.
	 Then $\mathbb{E}(f\vert \mathcal{C})(y)=\mu_{y}(\{x_{y}\})$ and $T^{n}\mathbb{E}(f\vert \mathcal{C})(y)=\mu_{T^{n}y}(\{x_{T^{n}y}\})$. Moreover, $\mathbb{E}(f\cdot T^{n}f\vert \mathcal{C})(y)$ equals to $\mu_{y}(\{x_{y}\})$ if $T^{n}x_{y}=x_{T^{n}y}$ and equals to 0 otherwise.

	Suppose that there exist $\e>0$ and  $B\in\mathcal{C}$ with $\nu(B)>0$ such that for all $y\in B$, $\e<\mu_{y}(\{x_{y}\})<1-\e$. Let $n\in\mathbb{N}$ be such that $\nu(B\cap T^{-n}B)>\nu(B)/2$.
	Then for all $y\in B\cap T^{-n}B$, the difference 
	$$\vert\mathbb{E}(f\cdot T^{n}f\vert \mathcal{C})(y)-\mathbb{E}(f\vert \mathcal{C})(y)\cdot T^{n}\mathbb{E}(f\vert \mathcal{C})(y)\vert$$
	is either $\mu_{y}(\{x_{y}\})\mu_{T^{n}y}(\{x_{T^{n}y}\})$ or $\mu_{y}(\{x_{y}\})(1-\mu_{T^{n}y}(\{x_{T^{n}y}\}))$, both of which are at least $\e^{2}$. This implies that for such $n\in\mathbb{N}$,
	\begin{equation}\nonumber
	\begin{split}
	&\quad\int_{Y}\Bigl\vert\mathbb{E}(f\cdot T^{n}f\vert \mathcal{C})(y)-\mathbb{E}(f\vert \mathcal{C})(y)\cdot T^{n}\mathbb{E}(f\vert \mathcal{C})(y)\Bigr\vert\,d\nu(y)
	\\&\geq \e^{2}\mu(B\cap T^{-n}B)>\e^{2}\nu(B)/2>0.
	\end{split}
	\end{equation}
	Since the set of $n\in\mathbb{N}$ such that $\nu(B\cap T^{-n}B)>\nu(B)/2$ has positive density in $\mathbb{N}$, this is a contradiction to the fact that $X$ is a weakly mixing extension of $Y$. Since $\mu_{y}(\{x_{y}\})>0$ for $\nu$-a.e. $y\in Y$, this contradiction implies that $\mu_{y}(\{x_{y}\})=1$ for  $\nu$-a.e. $y\in Y$. It follows that  for $\nu$-a.e. $y\in Y$, $\mu_{y}$ is supported on a single point, which contradicts to the fact that $X$ is a non-trivial extension of $Y$.
	We are done.
\end{proof}

\begin{lem}\label{roh3}
	Let $(Y,\mathcal{C},\nu,T)$ be a free, invertible, ergodic probability measure preserving system, and let $(X,\mathcal{B},\mu,T)$ be a nontrivial relatively weakly mixing extension of $(Y,\mathcal{C},\nu,T)$. For every $\tau$-regular set $C\in\mathcal{B}$ with $\mu(C)<\frac{\tau}{d+1}$, every $N\in\N,\e>0$ and every $0<a<1-\tau^{-1}\mu(C)$, there exists a $\tau$-regular set $A\in\mathcal{B}$ such that
	\begin{itemize}
		\item $\mu(A)=\tau a$;
		\item $A\cap C=\emptyset$;
		\item $A\cup C$ is $\tau$-regular;
		\item $\mu(A\cap T^{-p_{1}(n)}A\cap\dots\cap T^{-p_{d}(n)}A)>(1-\e)\Bigl(1-\frac{(d+1)\mu(C)}{\tau}\Bigr)\mu(A)$
		for all $0\leq n<N$.
	\end{itemize}
	In this case, we say $A$ is \emph{$(C,a,\e,N)$-$\tau$-good}.
\end{lem}
\begin{proof}
	Let $\mu=\int_{Y}\mu_{y}\,d\nu(y)$ be the disintegration of $\mu$ with respect to $\nu$. By Lemma \ref{lem:00}, $\mu_{y}$ is atomless for $\nu$-a.e. $y\in Y$. 
	
	We may assume without loss of generality that $p_{1}(n),\dots,p_{d}(n)$ are monotone for $n>0$.
	Denote $L=\sum_{i=1}^{d}\vert p_{i}(N)\vert$. Let $M>\lceil\frac{1}{\e}\rceil$ be such that $\pi(C)$ is $(ML,\e)$-uniform on $Y$, whose existence is guaranteed by Lemma \ref{lem:uniform}. Let $B$ be the base of a Rohlin tower on $Y$ of height $ML$ such that
	$$\nu(\cup_{i=0}^{ML-1}T^{-i}B)>1-\e.$$
	
	Let $\pi\colon X\to Y$ be the factor map.
	Let $I\subseteq B$ be such that the set
	$$A'=\{T^{-i}y\colon y\in I, 0\leq i<ML, T^{-i}y\notin \pi(C)\}$$
	is of measure $\nu(A')=a$ (this can be achieved since $Y$ ergodic and free and thus atomless). Denote
	$$I_{i}=\{y\in I\colon T^{-i}y\in \pi(C)\}.$$
	Let $J\in\mathcal{B}$ be an arbitrary $\tau$-regular set with $\pi(J)=I$ (we can do so since $\mu_{y}$ is atomless for $\nu$-a.e. $y\in Y$) and let $A_{i}=T^{-i}(J\backslash \pi^{-1}(I_{i})).$
	
	We claim that the set $A=\bigcup_{i=0}^{ML-1}A_{i}$ is $(C,a,\e,N)$-$\tau$-good.
	By the construction of $A_{i}$, $A_{i}\cap C=\emptyset$ and so $A\cap C=\emptyset$ (in fact we have that $\pi(A)\cap \pi(C)=\emptyset$).
	Since $J$ is $\tau$-regular, so are $J\backslash\pi^{-1}(I_{i})$ and $A_{i}$. Since $\pi(J)=I$, all of $A_{i}$ are pairwise disjoint. So $A$ is $\tau$-regular. Since $\pi(A)\cap \pi(C)=\emptyset$ and $C$ is $\tau$-regular, we have that $A\cup C$ is $\tau$-regular.
	
	Note that $$\nu(\pi(A))=\sum_{i=0}^{ML-1}\nu(\pi(A_{i}))=\nu(A').$$ Since $A$ is $\tau$-regular, $\mu(A)=\tau\cdot\nu(\pi(A))=\tau a$. 
	
	Let $W=\sum_{i=0}^{ML-1}T^{i}J$. We have that
	\begin{equation}\nonumber
	\begin{split}
	&\quad\mu(W\backslash A)=\tau\cdot\nu(\pi(W)\backslash \pi(A))=\tau\cdot\nu\Bigl((\bigcup_{i=0}^{ML-1}T^{-i}I)\backslash \pi(A)\Bigr)
	\\&=\tau\cdot\nu\Bigl((\bigcup_{i=0}^{ML-1}T^{-i}I\cap\pi(C)\Bigr)\geq(1-\e) \tau\nu(\pi(C))\nu(\bigcup_{i=0}^{ML-1}T^{-i}I)=(1-\e)\mu(C)\mu(W)/\tau,
	\end{split}
	\end{equation}
	where in the last inequality we used the fact that $\pi(C)$ is $(ML,\e)$-uniform. So
	$$\mu(A)\leq (1-\frac{(1-\e)\mu(C)}{\tau})\mu(W).$$
	For $0\leq n<N$,
	\begin{equation}\nonumber
	\begin{split}
	&\quad\mu(A\cap T^{-p_{1}(n)}A\cap\dots\cap T^{-p_{d}(n)}A)
	\\&>\mu(W\cap T^{-p_{1}(n)}W\cap\dots\cap T^{-p_{d}(n)}W)-\frac{(d+1)(1-\e)\mu(C)\mu(W)}{\tau}
	\\&>(\mu(W)-\frac{\sum_{i=1}^{d}\vert p_{i}(n)\vert}{ML})-\frac{(d+1)(1-\e)\mu(C)\mu(W)}{\tau}
	\\&>(1-\e)\mu(W)-\frac{(d+1)(1-\e)\mu(C)\mu(W)}{\tau}
	\\&=(1-\e)\Bigl(1-\frac{(d+1)\mu(C)}{\tau}\Bigr)\mu(W)
	\geq(1-\e)\Bigl(1-\frac{(d+1)\mu(C)}{\tau}\Bigr)\mu(A).
	\end{split}
	\end{equation}
\end{proof}

\begin{proof}[Proof of Proposition \ref{thm:rwm}]
	

	Let $0<a_{i},\e_{i}<1, i\in\mathbb{N}$ to be chosen later. Since $X$ is a weakly mixing extension of $Y$, for every $f\in L^{\infty}(\mu)$, by Proposition 2.3 of \cite{BL}, we have that
		\begin{equation}\nonumber
		\begin{split}
		&\quad\lim_{N\to\infty}\frac{1}{N}\sum_{n=0}^{N-1}\Bigl\vert\int_{X} \prod_{i=0}^{d} T^{p_{i}(n)}f\, d\mu-\int_{X}\prod_{i=0}^{d}  \E(T^{p_{i}(n)}f\vert\mathcal{C})\, d\mu\Bigr\vert
		\\&=\lim_{N\to\infty}\frac{1}{N}\sum_{n=0}^{N-1}\Bigl\vert\int_{X} \prod_{i=0}^{d} \E(T^{p_{i}(n)}\vert\mathcal{C})\, d\mu-\int_{X}\prod_{i=0}^{d}  \E(T^{p_{i}(n)}f\vert\mathcal{C})\, d\mu\Bigr\vert
		\\&\leq\lim_{N\to\infty}\frac{1}{N}\sum_{n=0}^{N-1}\int_{X} \Bigl\vert\prod_{i=0}^{d} \E(T^{p_{i}(n)}f\vert\mathcal{C})\, d\mu-\int_{X}\prod_{i=0}^{d}  \E(T^{p_{i}(n)}f\vert\mathcal{C})\Bigr\vert\, d\mu=0.
		\end{split}
		\end{equation}	
	Since $$\Bigl\vert\int_{X}\prod_{i=0}^{d}  \E(T^{p_{i}(n)}f\vert\mathcal{C})\, d\mu\Bigl\vert\geq \Vert f\Vert_{\infty}^{d+1},$$ for every $\e>0$, the set 
	$$\Bigl\vert\Bigl\{n\leq N\colon\Bigl\vert\frac{\int_{X} \prod_{i=0}^{d} T^{p_{i}(n)}f\, d\mu}{\int_{X}\prod_{i=0}^{d}  \E(T^{p_{i}(n)}f\vert\mathcal{C})\, d\mu}-1\Bigr\vert>\e\Bigr\}\Bigr\vert<\e N$$
	when $N$ is sufficiently large.
	
	 Let $a_{i}=\frac{a}{i(i+1)}$, where $a<\frac{1}{d+1}$.
	Let $A_{1}$ be an arbitrary $\tau$-regular set with $\mu(A_{1})=\tau a_{1}$. The existence of $A_{1}$ is guaranteed if $\pi$ is non-trivial. Let $f_{1}=\bold{1}_{A_{1}}$. Let $N_{1}\in\N$ be such that the cardinality of
	$$E_{1,N}:=\Bigl\vert\Bigl\{n\leq N\colon\vert\frac{\int_{X} \prod_{i=0}^{d} T^{p_{i}(n)}f_{1}\, d\mu}{\int_{X}\prod_{i=0}^{d}  \E(T^{p_{i}(n)}f_{1}\vert\mathcal{C})\, d\mu}-1\vert>\e_{1}\Bigr\}\Bigr\vert$$ is at most $\e_{1} N$ for all $N>N_{1}$.
	
	Suppose $A_{i}, N_{i}$ are chosen and $f_{i}=\bold{1}_{A_{i}}$ for all $i\leq k$. Denote $C_{j}=\bigcup_{i=1}^{j}A_{i}$ and $g_{j}=\bold{1}_{C_{j}}=\sum_{i=1}^{j}f_{i}$. Let $A_{k+1}$ be a $(C_{k},a_{k+1},\e_{k},N_{k})$-$\tau$-good set. The existence of $A_{k+1}$ is guaranteed by Lemma \ref{roh3} since $0<\sum_{i=1}^{\infty}\tau a_{i}=\frac{\tau}{d+1}$. Let $N_{k+1}>N_{k}$ be such that the cardinality of
	$$E_{k+1,N}:=\Bigl\vert\Bigl\{n\leq N\colon\vert\frac{\int_{X} \prod_{i=0}^{d} T^{p_{i}(n)}f_{k+1}\, d\mu}{\int_{X}\prod_{i=0}^{d}  \E(T^{p_{i}(n)}f_{k+1}\vert\mathcal{C})\, d\mu}-1\vert>\e_{k+1}\Bigr\}\Bigr\vert$$ is at most $\e_{k+1} N$ for all $N>N_{k+1}$.
	
	Let $A=\bigcup_{i=1}^{\infty}A_{i}$. We claim that $g=\bold{1}_{A}=\lim_{i\to\infty}g_{i}=\sum_{i=1}^{\infty}f_{i}$ satisfies the condition of the theorem. Suppose that $N_{k}\leq N<N_{k+1}$ for some $k\in\mathbb{N}$ with $(d+1)\tau^{d}<\frac{k+1}{k+2}$. If $n\notin E_{k,N}$, then
	\begin{equation}\nonumber
	\begin{split}
	&\quad\int_{X} \prod_{i=0}^{d} T^{p_{i}(n)}g\, d\mu
	=\mu(A\cap T^{-p_{1}(n)}A\cap\dots\cap T^{-p_{d}(n)}A)
	\\&\geq \mu(C_{k}\cap T^{-p_{1}(n)}C_{k}\cap\dots\cap T^{-p_{d}(n)}C_{k})+\sum_{i=2}^{\infty}\mu(A_{k+i}\cap T^{-p_{1}(n)}A_{k+i}\cap\dots\cap T^{-p_{d}(n)}A_{k+i})
	\\&>(1-\e_{k})\int_{X}\prod_{i=0}^{d}  \E(T^{p_{i}(n)}g_{k}\vert\mathcal{C})\, d\mu+\sum_{i=2}^{\infty}(1-\e_{k+i})\Bigl(1-\frac{(d+1)\mu(A)}{\tau}\Bigr)\mu(A_{k+i})
	\\&>(1-\e_{k})\int_{X}\prod_{i=0}^{d}  \E(T^{p_{i}(n)}g_{k}\vert\mathcal{C})\, d\mu+(1-\e_{k+1})\Bigl(1-(d+1)a\Bigr)\sum_{i=2}^{\infty}\mu(A_{k+i})
	\\&=(1-\e_{k})\int_{X}\prod_{i=0}^{d}  \E(T^{p_{i}(n)}g_{k}\vert\mathcal{C})\, d\mu+(1-\e_{k+1})\Bigl(1-(d+1)a\Bigr)\frac{a}{k+2}.
	\end{split}
	\end{equation}
	
	By the constructions, $C_{k+1}$ is $\tau$-regular and so
	$0\leq \E(g_{k+1}\vert Y)(y), \E(g\vert Y)(y)\leq \tau$ for $\nu$-a.e. $y\in Y$. So
	\begin{equation}\nonumber
	\begin{split}
	&\quad\int_{X}\prod_{i=0}^{d}  \E(T^{p_{i}(n)}g\vert\mathcal{C})\, d\mu-\int_{X}\prod_{i=0}^{d}  \E(T^{p_{i}(n)}g_{k}\vert\mathcal{C})\, d\mu
	\\&\leq\sum_{j=0}^{d}\int_{X}\E(T^{p_{j}(n)}(g-g_{k})\vert\mathcal{C})\prod_{i\neq j}  \E(T^{p_{i}(n)}g\vert\mathcal{C})\, d\mu
	\\&\leq\sum_{j=0}^{d}\tau^{d}\int_{X}\E(g-g_{k}\vert\mathcal{C})\, d\mu
	\\&=(d+1)\tau^{d}\int_{X}(g-g_{k})\, d\mu
	=(d+1)\tau^{d}\cdot\frac{a}{k+1}           
	\end{split}
	\end{equation}
	So 
	\begin{equation}\nonumber
	\begin{split}
	&\quad\int_{X} \prod_{i=0}^{d} T^{p_{i}(n)}g\, d\mu-\int_{X}\prod_{i=0}^{d}  \E(T^{p_{i}(n)}g\vert\mathcal{C})\, d\mu
	\\&\geq(1-\e_{k+1})\Bigl(1-(d+1)a\Bigr)\frac{a}{k+2}-(d+1)\tau^{d}\cdot\frac{a}{k+1} -\e_{k}\int_{X}\prod_{i=0}^{d}  \E(T^{p_{i}(n)}g\vert\mathcal{C})\, d\mu.
	\end{split}
	\end{equation}
	The right hand side is positive if we pick $a$ sufficiently small, $\e_{k}$ decreasing to 0 sufficiently fast, and $k$ large enough (since $(d+1)\tau^{d}<\frac{k+1}{k+2}$).

	Since $\vert E_{k,N}\vert<\e_{k}N$ and $\e_{k}\to 0$, the set
	$$\Bigl\{n\in\mathbb{N}\colon\int_{X} \prod_{i=0}^{d} T^{p_{i}(n)}g\, d\mu>\int_{X}\prod_{i=0}^{d}  \E(T^{p_{i}(n)}g\vert\mathcal{C})\, d\mu\Bigr\}$$
	is of density 1.   
\end{proof}

\section{Over-independence for amenable actions}\label{sec:am}
In this section we address the over-independence phenomenon for measure preserving actions of countable amenable groups. 

A countable group $G$ is \emph{amenable} if there exists a sequence of finite sets $(F_{N})_{N\in\mathbb{N}}$ (called a \emph{F\o lner sequence}) such that for any finite set $K\subseteq G$, we have that 
$$\lim_{N\to\infty}\frac{\vert KF_{N}\Delta F_{N}\vert}{\vert F_{N}\vert}=0.$$


Let $G$ be a countable amenable group and $(X,\mathcal{B},\mu,(T_{g})_{g\in G})$ be a probability measure preserving system. 
\begin{itemize}
	\item $(X,\mathcal{B},\mu,(T_{g})_{g\in G})$ is \emph{mixing} if for all $A\in\mathcal{B}$, we have $$\lim_{g\to\infty}\mu(A\cap T_{g}A)=\mu(A)^{2}$$
	(meaning that for any $\epsilon>0$, the set 
	$\{g\in G\colon \vert \mu(A\cap T_{g}A)-\mu(A)^{2}\vert>\epsilon\}$
	is finite);
	\item $(X,\mathcal{B},\mu,(T_{g})_{g\in G})$ is \emph{weakly mixing} if for any $A\in\mathcal{B}$ and any F\o lner sequence $(F_{N})_{N\in\mathbb{N}}$, $$\lim_{N\to\infty}\frac{1}{\vert F_{N}\vert}\sum_{g\in F_{N}}\vert\mu(A\cap T_{g}A)-\mu(A)^{2}\vert=0;$$
	\item $(X,\mathcal{B},\mu,(T_{g})_{g\in G})$ is \emph{ergodic} if for any $A\in\mathcal{B}$ and any F\o lner sequence $(F_{N})_{N\in\mathbb{N}}$, we have $$\lim_{N\to\infty}\frac{1}{\vert F_{N}\vert}\sum_{g\in F_{N}}\mu(A\cap T_{g}A)=\mu(A)^{2}.$$
\end{itemize}

We have the following results:

\begin{thm}[Over-independence]\label{thm:ma} Let $G$ be a countable amenable group and $(X,\mathcal{B},\mu,(T_{g})_{g\in G})$ be a mixing probability measure preserving system. Then
	there exists $A\in\mathcal{B}$ such that 
	$$\mu(A\cap T_{g}A)>\mu(A)^{2}$$
	for all $g\in G$.\footnote{A generalization of Theorem \ref{thm:hm}  also holds for actions of amenable groups which are mixing of order $d$. We omit the proof.}
\end{thm}

\begin{thm}[Density-1 over-independence]\label{thm:wa} Let $G$ be a countable amenable group and $(F_{N})_{N\in\mathbb{N}}$ be a F\o lner sequence of $G$. Let $(X,\mathcal{B},\mu,(T_{g})_{g\in G})$ be a  weakly mixing probability measure preserving system. Then
	there exists $A\in\mathcal{B}$ such that the set
	$$\{g\in G\colon\mu(A\cap T_{g}A)>\mu(A)^{2}\}$$
	is of density 1.
\end{thm}

We say that a system  $(X,\mathcal{B},\mu,(T_{g})_{g\in G})$ is \emph{free} if $T_{g}$ is not the identity map for all $g\in G, g\neq e_{G}$.
For Ces\`aro over-independence we have:  

\begin{thm}[Ces\`aro over-independence]\label{thm:ea}  Let $G$ be a countable amenable group and $(F_{N})_{N\in\mathbb{N}}$ be a F\o lner sequence of $G$. Let $(X,\mathcal{B},\mu,(T_{g})_{g\in G})$ be an ergodic and free probability measure preserving system. Then
	there exists $A\in\mathcal{B}$ such that 
	$$\frac{1}{\vert F_{N}\vert}\sum_{g\in F_{N}}\mu(A\cap T_{g}A)>\mu(A)^{2}$$
	for all $N\in\mathbb{N}$.
\end{thm}

Theorems \ref{thm:ma}, \ref{thm:wa} and \ref{thm:ea} can be proved by adjusting the method in the previous sections. 
The main novelty is the use of the more sophisticated Ornstein-Weiss Rohlin tower theorem for amenable actions instead of the classical Rohlin's result.

Since the proofs of Theorems \ref{thm:ma}, \ref{thm:wa} and \ref{thm:ea} are similar to those of Theorems \ref{thm:hm}, \ref{dense-one-under} Part (ii) and \ref{thm:e0}, respectively,
we will only prove Theorem \ref{thm:ma} in this paper and leave the proofs of Theorem  \ref{thm:wa} and \ref{thm:ea} to the interested reader.

\subsection{Ornstein-Weiss Rohlin tower theorem for amenable actions}
We start with recalling some definitions from \cite{OW}. Let $G$ be a countable amenable group. Let $K\subset G$ be finite and let $\delta>0$. We say that a finite subset $A\subset G$ is \emph{$(K,\delta)$-invariant} if 
$$\frac{\vert\{ g\in G\colon Kg\cap A\neq\emptyset \text{ and } Kg\cap (G\backslash A)\neq\emptyset\}\vert}{\vert A\vert}<\delta.$$
The set $\{ g\in G\colon Kg\cap A\neq\emptyset \text{ and } Kg\cap (G\backslash A)\neq\emptyset\}$ is called the \emph{$K$-boundary} of $A$.

For $H,B\subseteq G$, if the sets $hB, h\in H$ are pairwise disjoint, we say that $HB$ is an \emph{$H$-tower} with \emph{base} $B$.

A collection of subsets $A_{1},\dots,A_{k}$ of $G$ is \emph{$\epsilon$-disjoint} if there exist $A'_{i}\subseteq A_{i}$ such that $\vert A'_{i}\vert>(1-\epsilon)\vert A_{i}\vert, A'_{i}\cap A'_{j}=\emptyset$ for all $1\leq i\leq k, i\neq j$.
We say that a collection of subsets $A_{1},\dots,A_{k}$ of $G$ \emph{$\alpha$-cover} a subset $D$ of $G$ if $\vert D\cap (\cup_{i=1}^{k}A_{i})\vert\geq \alpha\vert D\vert$.

We say that a finite collection of subsets $\{G_{1},\dots,G_{N}\}$ of $G$  \emph{$\epsilon$-quasi-tile} $G$ if $e_{G}\in G_{1}\subset G_{2}\subset\dots\subset G_{N}$ and for any finite set $D\subseteq G$, there exist sets $C_{i}, 1\leq i\leq N$ such that 
\begin{itemize}
	\item for fixed $i$, all the sets $G_{i}c, c\in C_{i},$ are $\epsilon$-disjoint;
	\item for $i\neq j$, $G_{i}C_{i}\cap G_{j}C_{j}=\emptyset$;
	\item the sets $G_{i}C_{i}, 1\leq i\leq N$, $(1-\epsilon)$-cover $D$.
\end{itemize}

\begin{thm}[\cite{OW}, p.24]\label{thm:1.6} Given $\epsilon>0$, there is an $N=N(\epsilon)$ such that for every countable amenable group $G$, every finite $K\subseteq G$ and $\delta>0$, there are subsets $\{T_{1},\dots,T_{N}\}$ of $G$ that are $(K,\delta)$-invariant and $\epsilon$-quasi-tile $G$.
\end{thm}

\begin{thm}[\cite{OW}, p.59]\label{thm:2.5}
	Let $G$ be a countable amenable group and $\epsilon>0$. Then there exist a finite set $K_{0}=K_{0}(\e)\subseteq G$ and $\delta_{0}=\delta_{0}(\e)>0$ such that for any finite set $K_{0}\subseteq K\subseteq G$ and $0<\delta<\delta_{0}$, and any $\{G_{1},\dots,G_{k}\}$ which are $(K,\delta)$-invariant subsets of $G$ that $\e$-quasi-tile $G$, 
	there exist $V_{i}^{j}\in\mathcal{B}, 1\leq j\leq L_{i},1\leq i\leq k$ such that
	\begin{itemize}
		\item each $R^{j}_{i}:=G_{i}V^{j}_{i}, 1\leq j\leq L_{i}$ is a $G_{i}$-tower;
		\item For each $1\leq i\leq k$, the sets $\{R^{1}_{i},\dots,R^{L_{i}}_{i}\}$ are $\e$-disjoint;
		\item For $i\neq i'$ and every $j,j'$, we have that $R^{j}_{i}\cap R^{j'}_{i'}=\emptyset$;
		\item $\mu(\bigcup_{i=1}^{k}\bigcup_{j=1}^{L_{i}}R^{j}_{i})>1-\e$.
	\end{itemize}
\end{thm}

\subsection{Over-independence for mixing actions of amenable groups}   

\begin{lem}\label{roh4}
	 Let $G$ be a countable amenable group and $(F_{N})_{N\in\mathbb{N}}$ be a F\o lner sequence of $G$. Let $(X,\mathcal{B},\mu,(T_{g})_{g\in G})$ be an ergodic and free probability measure preserving system. For every $C\in\mathcal{B}, N\in\N$, $0<a<1-\mu(C)$, if $\e>0$ is sufficiently small depending only on $a$, then there exists $A\in\mathcal{B}$ such that $\mu(A)=a, A\cap C=\emptyset,$ and
	$$\mu((C\cup A)\cap T_{c_{1}}A\cap\dots\cap T_{c_{d}}A)>(1-\e)\mu(A)$$
	for all $c_{i}\in F_{N}$. 
\end{lem}
\begin{proof}
		Fix $\epsilon>0$ and let $K_{0}(\e)$ and $\delta_{0}(\e)$ be chosen as in Theorem \ref{thm:2.5} for this $\epsilon$. Pick any $K_{0}(\e)\subseteq F_{N}$ and $\delta>\delta_{0}(\e)$.
	By Theorem \ref{thm:1.6}, there exist $(F_{N},\delta)$-invariant sets $\{G_{1},\dots,G_{k}\}$ which $\e$-quasi-tile $G$. By Theorem \ref{thm:2.5}, there exists $V_{i}^{j}\in\mathcal{B}, 1\leq j\leq L_{i},1\leq i\leq k$ such that
	\begin{itemize}
		\item each $R^{j}_{i}:=G_{i}V^{j}_{i}, 1\leq j\leq L_{i}$ is a $G_{i}$-tower;
		\item For each $1\leq i\leq k$, the sets $\{R^{1}_{i},\dots,R^{L_{i}}_{i}\}$ are $\e$-disjoint;
		\item For $i\neq i'$ and every $j,j'$, we have that $R^{j}_{i}\cap R^{j'}_{i'}=\emptyset$;
		\item $\mu(\bigcup_{i=1}^{k}\bigcup_{j=1}^{L_{i}}R^{j}_{i})>1-\e$.
	\end{itemize}
	Since $X$ is ergodic and free, it is also atomless. So there exists a set $I=\bigcup_{i=1}^{k}\bigcup_{j=1}^{L_{i}}I^{j}_{i}$ with $I^{j}_{i}\subseteq V^{j}_{i}$ for all $i,j$, such that the set
	$$A=\Bigl(\bigcup_{i=1}^{k}\bigcup_{j=1}^{L_{i}}G_{i}I^{j}_{i}\Bigr)\backslash C$$
	has the property that $\mu(A)=a$. We claim that this set satisfies the requirements stipulated in the formulation of Lemma \ref{roh4}. 
	
	Obviously $A\cap C=\emptyset$. 
	Let $U_{i}$ denote the $F_{N}$-boundary of $G_{i}$ for $1\leq i\leq k$.  
	Note that if $x\in g_{i}I^{j}_{i}\cap A$ for some $i,j$ and $g_{i}\in G_{i}\backslash U_{i}$, then $T_{g}x\in G_{i}I^{j}_{i}\subseteq A\cup C$ for all $g\in F_{N}$. So
	$$\mu((C\cup A)\cap T_{c_{1}}A\cap\dots\cap T_{c_{d}}A)>\mu(A)-\mu(A')$$
	for all $c_{i}\in F_{N}$, where 
	$$A':=\bigcup_{i=1}^{k}\bigcup_{j=1}^{L_{i}}U_{i}I^{j}_{i}.$$ 
	
	Since for each $1\leq i\leq k$, the sets $\{R^{1}_{i},\dots,R^{L_{i}}_{i}\}$ are $\e$-disjoint, by the fact that $G_{i}$ are $(F_{N},\epsilon)$-invariant, we have that
	$\mu(A')\leq 10\e$. Therefore, 
	$$\mu((C\cup A)\cap T_{c_{1}}A\cap\dots\cap T_{c_{d}}A)>\mu(A)-10\e>(1-\sqrt{\e})\mu(A)$$
	if $\e$ is sufficiently small depending only on $\mu(A)$. This finishes the proof.
\end{proof}

\begin{proof}[Proof of Theorem \ref{thm:ma}]
	
	Let $(F_{N})_{N\in\mathbb{N}}$ be a F\o lner sequence of $G$, and $0<a_{i},\e_{i}<1, i\in\mathbb{N}$ to be chosen later.
	Let $A_{1}$ be an arbitrary set with $\mu(A_{1})=a_{1}$. 
	Since the system is mixing, there exists $N_{1}\in\N$ such that 
	$$\vert \mu(A_{1}\cap T_{g}A_{1})-\mu(A_{1})^{2}\vert<\e_{1}\mu(A_{1})^{2}$$
	for all $g\notin F_{N_{1}}$.
	
	Suppose $A_{i}, N_{i}$ are chosen for all $i\leq k$. Denote $C_{j}=\bigcup_{i=1}^{j}A_{i}$. Let $A_{k+1}$ be
	such that $\mu(A_{k+1})=a_{k+1}, A_{k+1}\cap C_{k}=\emptyset,$ and
	$$\mu((C_{k}\cup A_{k+1})\cap T_{g}A_{k+1})>(1-\e_{k})\mu(A_{k+1})$$
	for all $g\in F_{N_{k}}$. The existence of $A_{k+1}$ is guaranteed by Lemma \ref{roh4} if $\e_{k}\ll a_{k+1}$ and $0<\sum_{i=1}^{\infty}a_{i}<1$. Let $N_{k+1}>N_{k}$ be such that 
	$$\vert\mu(C_{k+1}\cap T_{g}C_{k+1})-\mu(C_{k+1})^{2}\vert<\e_{k+1}\mu(C_{k+1})^{2}$$
	for all $g\notin F_{N_{k+1}}$.
	
	We claim that $A=\bigcup_{i=1}^{\infty}A_{i}$ satisfies the conditions of the theorem. 
	 If $g\in F_{N_{1}}$, then
	 \begin{equation}\nonumber
	 \begin{split}
	 \mu(A\cap T_{g}A)
	 \geq \sum_{k=2}^{\infty}\mu(A\cap T_{g}A_{k})
	 \geq \sum_{k=2}^{\infty}(1-\epsilon_{1})\mu(A_{k})=(1-\epsilon_{1})a/2>a^{2}=\mu(A)^{2},
	 \end{split}
	 \end{equation}
	 provided that $a$ is sufficiently small and $\e_{1}<1/2$.
	
	Now suppose that $g\in F_{N_{k+1}}\backslash F_{N_{k}}$ for some $k>0$. Then
	\begin{equation}\nonumber
	\begin{split}
	&\mu(A\cap T_{g}A)
	\geq \mu(C_{k+1}\cap T_{g}C_{k+1})+\sum_{i=2}^{\infty}\mu(C_{k+i}\cap T_{g}A_{k+i})
	\\&>(1-\e_{k+1})\mu(C_{k+1})^{2}+\sum_{i=2}^{\infty}(1-\e_{k+i})\mu(A_{k+i})
	\\&=(1-\e_{k+1})(a_{1}+\dots+a_{k+1})^{2}+\sum_{i=2}^{\infty}(1-\e_{k+i})a_{k+i}.
	\end{split}
	\end{equation}
	If we pick $a_{i}=\frac{a}{i(i+1)}$, $a$ sufficiently small, and $\e_{i}$ decreasing to 0 sufficiently fast, then 
	\begin{equation}\nonumber
	\begin{split}
	&\quad(1-\e_{k+1})(a_{1}+\dots+a_{k+1})^{2}+\sum_{i=2}^{\infty}(1-\e_{k+i})a_{k+i}
	\\&>(1-\e_{k+1})((a-\frac{a}{k+2})^{2}+\frac{a}{k+2})>a^{2}=\mu(A)^{2}.
	\end{split}
	\end{equation}  
	This finishes the proof.
\end{proof}

\end{document}